\documentclass[11pt,english]{amsart}
\usepackage[T1]{fontenc}
\usepackage[latin1]{inputenc}
\usepackage{textcomp}
\usepackage{amsthm}
\usepackage{amstext}
\usepackage{a4wide}
\usepackage{amssymb}

\makeatletter

\providecommand{\tabularnewline}{\\}

\numberwithin{equation}{section}
\numberwithin{figure}{section}
\theoremstyle{plain}
\newtheorem{thm}{Theorem}
  \theoremstyle{plain}
  \newtheorem{prop}[thm]{Proposition}
  \theoremstyle{plain}
  \newtheorem{cor}[thm]{Corollary}
  \theoremstyle{plain}
  \newtheorem{lem}[thm]{Lemma}
  \theoremstyle{remark}
  \newtheorem{rem}[thm]{Remark}

\@ifundefined{definecolor}
 {\@ifundefined{definecolor}
 {\usepackage{color}}{}
}{}
\makeatother

\makeatother

\usepackage{babel}

\makeatother

\usepackage{babel}

\begin{document}
\global\long\global\long\def\Alb{{\rm Alb}}
 \global\long\global\long\def\Jac{{\rm Jac}}
 \global\long\global\long\def\Hom{{\rm Hom}}
 \global\long\global\long\def\End{{\rm End}}
 \global\long\global\long\def\aut{{\rm Aut}}
 \global\long\global\long\def\NS{{\rm NS}}
 \global\long\global\long\def\SSm{{\rm S}}
 \global\long\global\long\def\psl{{\rm PSL}}
 \global\long\global\long\def\CC{\mathbb{C}}
 \global\long\global\long\def\BB{\mathbb{B}}
 \global\long\global\long\def\PP{\mathbb{P}}
 \global\long\global\long\def\QQ{\mathbb{Q}}
 \global\long\global\long\def\RR{\mathbb{R}}
 \global\long\global\long\def\FF{\mathbb{F}}
 \global\long\global\long\def\DD{\mathbb{D}}
 \global\long\global\long\def\NN{\mathbb{N}}
 \global\long\global\long\def\ZZ{\mathbb{Z}}
 \global\long\global\long\def\HH{\mathbb{H}}
 \global\long\global\long\def\gal{{\rm Gal}}
 \global\long\global\long\def\OO{\mathcal{O}}
 \global\long\global\long\def\oo{{\scriptstyle {\mathcal{O}}}}
 \global\long\global\long\def\pP{\mathfrak{p}}
 \global\long\global\long\def\Sl{{\rm SL}}
 \global\long\global\long\def\Gl{{\rm GL}}
 \global\long\global\long\def\pgl{{\rm PGL}}
 \global\long\global\long\def\Nrd{{\rm Nrd}}
 \global\long\global\long\def\Trd{{\rm Trd}}

\title{Automorphisms and quotients of quaternionic fake quadrics}

\author{Amir {D}\v{z}ambi\'{c}, Xavier Roulleau}
\begin{abstract}
A fake quadric is a smooth surface of general type with the same invariants
as the quadric in $\mathbb{P}^{3}$, i.e. $c_{1}^{2}=8,\, c_{2}=4$
and $q=p_{g}=0$. We study here quaternionic fake quadrics i.e. fake
quadrics constructed arithmetically by using quaternion algebras
over real quadratic number fields. We provide examples of quaternionic fake
quadrics $X$ with a non-trivial automorphism group and compute the
invariants of the minimal desingularisation of the quotient of $X$
by this group. In that way we obtain minimal surfaces $Z$ of general
type with $q=p_{g}=0$ and $K^{2}=4$ or $2$ which contain the maximal
number of disjoint $(-2)$-curves. We then prove that if a surface
of general type has the same invariant as $Z$ and same number of
$(-2)$-curves, then we can construct geometrically a surface of general
type with $c_{1}^{2}=8,\ c_{2}=4$. 
\end{abstract}
\maketitle
\global\long\global\long\def\thethm{\Alph{thm}}



\section{Introduction}

A \emph{fake quadric} is a smooth minimal surface of general type
with the same numerical invariants as the quadric $\mathbb{P}^{1}\times\mathbb{P}^{1}$
i.e. with Chern numbers $c_{1}^{2}=8,\, c_{2}=4$ and vanishing geometric
genus $p_{g}=0$. Two classes of examples of such surfaces are known
and these two classes are both quotients of $\mathbb{H}\times\mathbb{H}$,
where $\mathbb{H}$ is the upper-half plane, by a cocompact torsion
free lattice $\Gamma\subset Aut(\mathbb{H})\times Aut(\mathbb{H})$. In
other words, their universal cover is always $\mathbb{H}\times\mathbb{H}$.

The first class of fake quadrics consists of surfaces $X=\Gamma\backslash\mathbb{H}\times\mathbb{H}$
such that the group $\Gamma$ is reducible. By reducible we mean that
there exists a subgroup of finite index $\Gamma'=\Gamma_{1}\times\Gamma_{2}$
of $\Gamma$ such that the group $\Gamma_{i}$ acts on $\mathbb{H}$
and $C_{i}=\mathbb{H}/\Gamma_{i}$ is a smooth algebraic curve. This
case is now well understood and the full classification of these fake
quadrics, called \emph{surfaces isogenous to a higher product}, has
been achieved in \cite{bauercatanesegrunewald} by Bauer, Catanese
and Grunewald. In practice, this classification and construction is
done geometrically by classifying triples $(C_{1},C_{2},G)$ of two
smooth curves $C_{i}$ of general type and an automorphism group $G$,
such that $G$ acts freely on the surface $C_{1}\times C_{2}$ and
the quotient $(C_{1}\times C_{2})/G$ has the asked invariants.

In this paper we will focus on fake quadrics of the second class,
that we call \emph{quaternionic fake quadrics}. These fake quadrics
are quotients of $\HH\times\HH$ by cocompact irreducible lattices
$\Gamma$ in $Aut(\mathbb{H}) \times Aut(\mathbb{H})$. The lattice $\Gamma$
is then arithmetic by a theorem of Margulis and is defined by an indefinite
quaternion algebra over a totally real number field.

The first quaternionic fake quadrics have been constructed by Shavel
\cite{Shavel78} in 1978. We know that these surfaces are rigid and
thus that there are only a finite number of them, but at the moment
we do not have a complete list of all these surfaces. We have a list
of commensurability classes of fake quadrics defined by quaternion algebras
over quadratic fields \cite{dz11}.

The situation for quaternionic fake quadrics is very similar to the
case of fake projective planes which are surfaces of general type
with the same numerical invariants as the projective plane. Fake projective
planes are all quotients of the $2$-dimensional complex unit ball
$\mathbb{B}^{2}$ by cocompact arithmetic lattices $\Gamma\subset PU(2,1)$.
This provides an arithmetic construction of these surfaces, but it
is generally not easy to handle and construct these surfaces geometrically,
e.g. as a quotient or ramified cover of some known surfaces.

In order to remedy at this situation, in \cite{Keum11}, \cite{Keum},
\cite{Keum2}, Keum studied quotients $Z$ of fake projective planes
by groups of automorphisms. In this way, he obtained surfaces of general
type with geometric genus $p_{g}=0$ and was able to rebuild a fake
projective plane by only knowing the properties of the quotient surface
$Z$.

The aim of this paper is to study automorphisms of quaternionic fake
quadrics and the quotients of these surfaces by groups of automorphisms.
The computations in \cite{dz11} leads us to the conjecture that the
order of the automorphism group is less or equal $24$ (see Section
\ref{other}).

The first main result we obtain is the following: 
\begin{thm}
Let $X=\Gamma\backslash\HH\times\HH$ be a quaternionic fake quadric.
An automorphism of $X$ has only finitely many fixed points. There
exist fake quaternionic quadrics $X$ with automorphism group isomorphic
to\[
\mathbb{Z}/2\mathbb{Z},\,(\mathbb{Z}/2\mathbb{Z})^{2},\,\mathbb{D}_{4},\,\mathbb{D}_{6},\,\mathbb{D}_{8},\,\text{or}\ \mathbb{D}_{10},\]
 where $\mathbb{D}_{n}$ is the dihedral group with order $2n$.
\end{thm}
Let us remark that the knowledge of surfaces of general type with
$p_{g}=0$ and a large automorphism group can be interesting to check
wether the Bloch conjecture holds (see e.g.~\cite{Inose}). 

The second aim of this paper is to study the minimal desingularisation
of the quotient of a quaternionic fake quadric by a group of automorphisms,
in order to obtain new surfaces with $p_{g}=0$.
\begin{thm}
Let $X$ be a quaternionic fake quadric and $G$ a finite group of
automorphisms of $X$. The minimal desingularisation $Z$ of the quotient
$X/G$ has has the following numerical invariants:\\

\begin{tabular}{|c|c|c|c|c|c|}
\hline 
$G$  & $c_{1}^{2}(Z)$  & $c_{2}(Z)$  & Singularities on $X/G$  & Minimal  & $\kappa(Z)$\tabularnewline
\hline
\hline 
$\mathbb{Z}/2\mathbb{Z}$  & $4$  & $8$  & $4A_{1}$  & yes  & 2\tabularnewline
\hline 
$\mathbb{Z}/3\mathbb{Z}$  & $2$  & $10$  & $2A_{3,1}+2A_{2}$  & - & 2\tabularnewline
\hline 
$\mathbb{Z}/6\mathbb{Z}$  & $-4$  & $16$  & $2A_{6,1}+2A_{5}$  & no  & - \tabularnewline
\hline 
$\mathbb{Z}/8\mathbb{Z}$  & $-2$ & $14$ & $A_{8,3}+A_{8,5}$ & no & -\tabularnewline
\hline 
$\mathbb{Z}/10\mathbb{Z}$  & $-12$  & $24$  & $2A_{10,1}+2A_{9}$  & no  & - \tabularnewline
\hline 
$(\mathbb{Z}/2\mathbb{Z})^{2}$  & $2$  & $10$  & $6A_{1}$  & yes  & 2\tabularnewline
\hline 

$\mathbb{D}_{4}$ & $0$ & $12$ & $4A_{1}+A_{4,3}+A_{4,1}$ & no & $\geq1$\tabularnewline
\hline 
$\mathbb{D}_{8}$ & $-1$ & $13$ & $4A_{1}+A_{8,3}+A_{8,5}$ & no & -\tabularnewline
\hline
\end{tabular}\\

\noindent Here, $\kappa$ indicates the Kodaira dimension of the surface
$Z$. 
\end{thm}
We obtain also results and restrictions for the groups $\mathbb{Z}/4\mathbb{Z}$,
$\mathbb{Z}/5\mathbb{Z}$ and $\mathbb{D}_{3}$. We note that the
surfaces general type we obtain have vanishing geometric genus. The
classification of surfaces with $p_{g}=0$ is not established and
we intend to compute the fundamental groups of our examples in a forthcoming
paper.

A curve $C$ on a surface is called nodal if $C\simeq\mathbb{P}^{1}$
and $C^{2}=-2$. A nodal curve is the resolution of a nodal singularity.
The surfaces $Z$ we obtain as quotient of a fake quadric by an automorphism
group $(\mathbb{Z}/2\mathbb{Z})^{n},\, n\in\{1,2\}$ have the maximum
number of nodal curves (so-called Miyaoka bound, see \cite{Miyaoka}). If minimal,
the surfaces obtained by quotient by the groups $\mathbb{Z}/3\mathbb{Z}$
and $\mathbb{D}_{3}$ have also the maximum number of quotient singularities.
As Keum did with fake planes, we can reverse the construction: 
\begin{prop}
Let $Z$ be a smooth minimal surface of general type with $q=p_{g}=0$.
\\
 a) Suppose that $c_{1}^{2}=4,\,2$, $Pic(Z)$ has no $2$-torsion,
and that there is a birational map $Z\to Y$ onto a surface containing
$8-c_{1}^{2}$ nodal singularities $A_{1}$. There exists a smooth
minimal surface of general type $S$ with invariants $c_{1}^{2}=2c_{2}=8$
and a $(\mathbb{Z}/2\mathbb{Z})^{m}$-cover $S\to Y$ ramified over
the nodes, with $m$ such that $2^{m}=\frac{8}{c_{1}^{2}}$.\\
 b) Suppose that $c_{1}^{2}=2$, $Pic(Z)$ has no $3$-torsion,
and that there is a birational map $Z\to Y$ onto a surface with $2A_{3,1}+2A_{2}$
singularities. There exist a smooth surface $S$ with invariants $c_{1}^{2}=2c_{2}=8$
and a $(\mathbb{Z}/3\mathbb{Z})$-cover $Z\to Y$ ramified over the
singularities of $Y$. 
\end{prop}
The proof of part $a)$ of this Proposition uses mainly the results
of Dolgachev, Mendes Lopes, Pardini (\cite{Dolgachev}) and illustrates
their theory. The proof of part b) is more original because it mixes
two types of singularities.

The paper is structured as follows: we begin to recall the known facts
on quaternionic fake quadrics, and on quotients of surfaces. We then
provide examples of fake quadrics having a large group of automorphisms,
compute the quotients surfaces and then reverse the construction on
the opposite direction : starting with a surface with the same invariants
as the quotient, we construct a surface with $c_{1}^{2}=2c_{2}=8$.

\textbf{Acknowledgements}. The authors warmly thank Fabrizio Catanese and Miles Reid for pointing out an
error in a previous version and H\aa{}kan Granath for his help to
correct it. We thank also Ingrid Bauer, Margarida Mendes Lopes and Rita Pardini
for many useful discussions. Part of this research was done during the
second author stay in Strasbourg University and in the Instituto Superior
Technico under grant FCT SFRH/BPD/72719/2010 and project Geometria
Algebrica PTDC/MAT/099275 /2008.

\global\long\global\long\def\thethm{\thesection.\arabic{thm}}

\section{Generalities on quaternionic fake quadrics}

\label{sec:generalities}Let us give a more detailed description of
quaternionic fake quadrics. First, recall that a lattice $\Gamma<\psl_{2}(\RR)\times\psl_{2}(\RR)\cong\aut\HH\times\aut\HH$
is irreducible if it is not commensurable with a product $\Gamma_{1}\times\Gamma_{2}$
of two discrete subgroups $\Gamma_{1},\Gamma_{2}\subset\psl_{2}(\RR)$.
Equivalently, the image of $\Gamma$ under the projection onto one
of the factors $\psl_{2}(\RR)$ is a dense subgroup of $\psl_{2}(\RR)$.
By a famous result of Margulis, an irreducible lattice $\Gamma$ in
$\psl_{2}(\RR)\times\psl_{2}(\RR)$ is an arithmetic group, and can
therefore be described in the following way:

There exists a totally real number field $k$ of degree $g=[k:\QQ]\geq2$
and a quaternion algebra $B=(\alpha,\beta)_{k}:=\langle1,i,j,ij\rangle_{k}$
with $i^{2}=\alpha\in k,j^{2}=\beta\in k,ij=-ji$, over $k$ such
that \begin{equation}
B\otimes_{\QQ}\RR=\prod_{\rho\in\Hom(k,\RR)}B^{\rho}\cong M_{2}(\RR)\times M_{2}(\RR)\times\underbrace{H_{\RR}\times\ldots\times H_{\RR}}_{g-2}.\label{eqquat}\end{equation}
 Here, $B^{\rho}=(\alpha^{\rho},\beta^{\rho})_{\RR}$ and $H_{\RR}=(-1,-1)_{\RR}$
denotes the skew field of Hamilton quaternions. Let $\oo_{k}$ be
the ring of integers of $k$ and $\OO$ a maximal order in $B$, i.e.
a subring of $B$ which is a full $\oo_{k}$-lattice in $B$. Finally,
let $\OO^{1}$ be the subgroup of all elements in $\OO$ of reduced
norm one. \\
 The isomorphism (\ref{eqquat}) induces an embedding of $\OO^{1}$
into $\Sl_{2}(\RR)\times\Sl_{2}(\RR)$ by taking the element $\gamma\in\OO^{1}$
to the pair $(\gamma^{\rho_{1}},\gamma^{\rho_{2}})\in\Sl_{2}(\RR)\times\Sl_{2}(\RR)$,
where $\gamma^{\rho_{i}}$ is the image of $\gamma$ in $B^{\rho_{i}}$.
The group $\OO^{1}$ then acts on $\HH\times\HH$ as a group of
fractional linear transformations. Namely, if $(z,w)\in\HH\times\HH$
is a point and an element $\gamma\in\OO^{1}$ is identified with
two matrices $\gamma^{\rho_{1}}$ and $\gamma^{\rho_{2}}\in\Sl_{2}(\RR)$,
then \[
\gamma(z,w)=(\gamma^{\rho_{1}}z,\gamma^{\rho_{2}}w).\]
 After dividing out by the ineffective kernel, one considers the group
\[
\Gamma_{\OO}^{1}=\OO^{1}/\{\pm1\}\subset\psl_{2}(\RR)\times\psl_{2}(\RR)\]
 and it can be proven that $\Gamma_{\OO}^{1}$ is an irreducible lattice
in $\psl_{2}(\RR)\times\psl_{2}(\RR)$ (see \cite{MatsushimaShimura}).
In general we say that a subgroup $\Gamma\subset\psl_{2}(\RR)\times\psl_{2}(\RR)$
is an \emph{arithmetic lattice} if there exists $k,B,\rho_{1},\rho_{2},\OO$
as above such that $\Gamma$ is commensurable with $\Gamma_{\OO}^{1}$.

Let $\Gamma$ be irreducible and $X_{\Gamma}:=\Gamma\backslash\HH\times\HH$
be the orbit space of the discontinuous action of $\Gamma$ on $\HH\times\HH$.
Then, there is a natural structure of compact algebraic surface on
$X_{\Gamma}$ and $X_{\Gamma}$ is smooth if and only if $\Gamma$
is torsion free. The numerical invariants of a smooth $X_{\Gamma}$
are computed in \cite{MatsushimaShimura}, see also \cite{Shavel78}.
It follows that $X_{\Gamma}$ is a fake quadric if and only if $c_{2}(X_{\Gamma})=4$
(see \cite{Shavel78}).

\section{Generalities on quotients of a surface}

In this section we recall results from the theory of singularities
and on the resolution of the quotient of a surface by a finite group. The
main reference for these topics is \cite{Barth}, see also \cite{Roulleau11}.

Let us denote by $G$ an automorphism group acting on $S$, by $X=S/G$
the quotient surface and by $\pi:Z\to S/G$ the minimal desingularisation
map. 
\begin{prop}
{[}Topological Lefschetz formula{]} Let $\sigma$ be an automorphism
acting on $S$ and $S^{\sigma}$ the fixed point set of $\sigma$.
We have \[
e(S^{\sigma})=\sum_{j=0}^{j=4}(-1)^{j}Tr(\sigma|H^{i}(S,\mathbb{Z})_{mt})\]
 where $H^{i}(S,\mathbb{Z})_{mt}$ is the group $H^{i}(S,\mathbb{Z})$
modulo torsion. 
\end{prop}
Note that for a fake quadric $S$ we have $q=p_{g}=0$, thus \[
H^{1}(S,\mathbb{Z})_{mt}=\{0\},\, H^{2}(S,\mathbb{Z})\otimes\mathbb{C}=H^{1}(S,\Omega_{S}).\]
 
\begin{cor}
\label{cor:lefschetz top} Let $S$ be a fake quadric and $\sigma$
an automorphism of order $n>1$ acting on $S$. We have $e(S^{\sigma})=2$
or $4$. If $\sigma=\tau^{2}$ for an automorphism $\tau$ (e.g. if
$n$ is prime to $2$), we have $e(S^{\sigma})=4$. \end{cor}
\begin{proof}
For a fake quadric, the space $H^{1}(S,\Omega_{S})$ is $2$-dimensional
and is generated by the classes of $2$ curves in the Néron-Severi
group. As an automorphism preserves the canonical divisor, the invariant
subspace of $H^{1}(S,\Omega_{S})$ is at least $1$ dimensional. Therefore
the trace of $\sigma$ on $H^{1}(S,\Omega_{S})$ is $2$ or $0$.
If we suppose that this action is not trivial, then $2$ divides the
order of $\sigma$, moreover we see that the action of $\sigma^{2}$
is always trivial.
\end{proof}
Let $\xi$ be a primitive $n^{th}$-root of unity. Let us recall that
for $1\leq q\leq n-1$ coprime to $n$, the quotient of $\mathbb{C}^{2}$
by the action of \[
(x,y)\to(\xi x,\xi^{q}y)\]
 has a unique singularity, called a $A_{n,q}$ singularity. For $n,m>0$
two numbers, we denote $[n,m]=n-\frac{1}{m}$. A $A_{n,q}$ singularity
is resolved by a chain of smooth rational curves $C_{1},\dots,C_{k}$
such that $C_{i}$ cuts $C_{i\pm1}$ for $2\leq i\leq k-1$ and $C_{i}^{2}=-n_{i}$
for integers $n_{i}\geq2$ determined by the relation: \[
\frac{n}{q}=[n_{1},[n_{2},\dots,[n_{k-1},n_{k}]]\dots].\]
 We denote classically $A_{n,n-1}=A_{n-1}$.

Let $S$ be a surface with $p_{g}=q=0$ and let $\sigma$ be an order
$n\geq2$ automorphism such that the fixed points of the $\sigma^{k},\, k=1,\dots,n-1$
are isolated. 
\begin{prop}
\label{pro:Lefschetz formula}(Holomorphic Lefschetz fixed point formula,
\cite{Atiyah} p. 567). Let $S^{\sigma}$ be the fixed point set of
$\sigma$. Then \[
1=\sum_{s\in S^{\sigma}}\frac{1}{\det(1-d\sigma|T_{S,s})},\]
 where $d\sigma_{s}\mid T_{S,s}$ denotes the action of $\sigma$
on the tangent space $T_{S,s}$. 
\end{prop}
Suppose moreover that the automorphism $\sigma$ has prime order $p$.
Let $\xi$ be a primitive $p^{th}$-root of unity. Let $r_{i}$ be
the number of isolated fixed points of $\sigma$ whose image in $S/\sigma$
are $A_{p,i}$ singularities. 
\begin{prop}
\label{prop Zhang formula}(Zhang's formula, \cite{Zhang} Lemma 1.6).
We have:\[
\sum_{i=1}^{i=p-1}r_{i}a_{i}(p)=1\]
 where \[
a_{i}(p)=\frac{1}{p-1}\sum_{j=1}^{j=p-1}\frac{1}{(1-\xi^{j})(1-\xi^{ij})}\]
 In particular, we have :\[
a_{1}(p)=\frac{5-p}{12},\, a_{2}(p)=\frac{11-p}{24},\, a_{3}(5)=\frac{1}{4},\, a_{4}(5)=\frac{1}{2}.\]

\end{prop}
Let $1\leq i<p$ and $1\leq k<p$ be such that $ik=1\,\mbox{mod}\, p$.
As $A_{p,i}=A_{p,k}$, the notations for $r_{i}$ and $r_{k}$ in
Zhang's Lemma can be confusing. However, as $a_{i}(p)=a_{k}(p)$,
there should be no trouble in taking the convention that $r_{i}+r_{k}$
is the total number of $A_{p,i}=A_{p,k}$ singularities, rather that
choosing a representative $i$ or $k$ for every such pair $(i,k)$.

Let us recall that an automorphism of a vector space is called a reflection
if all its eigenvalues but one are equal to $1$. Let $S$ be a surface
and $G$ an automorphism group acting on $S$. Suppose that for every
automorphism of $G$ the fixed point set is finite. Let $s$ be a
fixed point of $G$; recall (see \cite{Barth}): 
\begin{lem}
\label{lemme restriction sur les automor}The action of the group
$G$ on the tangent space $T_{S,s}$ is faithful and has no reflection. 
\end{lem}
In particular, if $G$ is cyclic of order $n$, the singularity type
of the image of the fixed point $s$ in the quotient $S/G$ is always
a $A_{n,q}$ with $q$ prime to $n$.


\begin{lem}
\label{lem:The-Euler-number}The Euler number of $S/G$ is given by
the formula \[
e(S/G)=\frac{1}{|G|}(e(S)+\sum_{n\geq2}(n-1)e(S_{n})),\]
 where $S_{n}=\{s\in S/|Stab(G,s)|=n\}$. The Euler number of the
minimal resolution $Z$ is the sum of $e(S/G)$ and the number of
irreducible components of the exceptional curves of the resolution
$\pi:Z\rightarrow S/G$. 
\end{lem}
Let $C_{1},\dots,C_{k}$ be the irreducible components of the one
dimensional fibers of $\pi:Z\to X=S/G$. We have the relations $K_{Z}=\pi^{*}K_{X}-\sum_{i=1}^{i=k}a_{i}C_{i}$,
for rational numbers $a_{i}$ such that $K_{Z}C_{k}=-2-C_{k}^{2}$
and $C_{k}\pi^{*}K_{X}=0$. \\
 Moreover we have the equality $K_{X}^{2}=\frac{K_{S}^{2}}{|G|}$ where $|G|$ is
the order of $G$. As $K_{S}$ ample, the canonical $\mathbb{Q}$-divisor
$K_{X}$ is ample and $\pi^{*}K_{X}$ is nef. We remark also that
$K_{Z}^{2}\leq K_{X}^{2}$. 
\begin{lem}
\label{lemme q=00003D00003D00003Dpg=00003D00003D00003D0}Let $S$
be a surface with $q=p_{g}=0$. The minimal resolution $Z$ of the
quotient of $S$ by a group $G$ has always $q=p_{g}=0$. 
\end{lem}
Suppose that $S$ is moreover minimal of general type and the fixed
points of automorphisms in $G$ are isolated, then:
\begin{lem}
\label{lem: when has general type}If $K_{Z}^{2}=0$, the surface
$Z$ has Kodaira dimension $\kappa\geq1$. If $K_{Z}^{2}>0$, the
surface $Z$ has Kodaira dimension $\kappa=2$.\end{lem}
\begin{proof}
(We follow the ideas from \cite{Keum}). The quotient surface has
$q=p_{g}=0$ and thus $\chi(\mathcal{O}_{Z})=1$. Let $m\geq1$ be
an integer, then $-mK_{Z}\pi^{*}K_{S/G}=-mK_{S/G}^{2}=-\frac{8}{|G|}m<0$,
therefore $H^{0}(Z,-mK_{Z})=\{0\}$ for every $m\geq1$. Let be $m\geq2$,
then by using Serre duality and Riemann-Roch: \[
H^{0}(Z,mK_{Z})=\chi(\mathcal{O}_{Z})+\frac{m(m-1)}{2}K_{Z}^{2}+h^{1}(Z,mK_{Z}).\]
 If $K_{Z}^{2}>0$, then immediately, $Z$ has general type. If $K_{Z}^{2}=0$,
the surface has $h^{0}(Z,2K_{Z})\not=0$ and cannot be rational by
Castelnuovo criterion. Moreover, as $\chi=1$ it cannot be a ruled
surface. Suppose that $Z$ is an Enriques surface. As $K_{Z}^{2}=0$,
it is a minimal surface, but this is impossible because $h^{0}(Z,3K_{Z})\not=0$
; therefore $\kappa>0$. 
\end{proof}
Let us now specialize to surfaces $X_{\Gamma}=\Gamma \backslash \mathbb{H}\times\mathbb{H}$
where $\Gamma$ is a cocompact and irreducible torsion-free lattice.
Let $\mu:\mathbb{H}\times\mathbb{H}\to\mathbb{H}\times\mathbb{H}$
be the involution exchanging the two factors. The group $\aut(\mathbb{H}\times\mathbb{H})$
is the semi-direct product of $\aut\mathbb{H}\times\aut\mathbb{H}$
by the group generated by $\mu$. Let $\Gamma<\aut\HH\times\aut\HH$
be a cocompact torsion-free lattice and $X_{\Gamma}=\Gamma\backslash\HH\times\HH$,
then, since $\HH\times\HH$ is the universal covering of $X_{\Gamma}$,
every automorphism $\sigma$ of $X_{\Gamma}$ lifts to an automorphism
$\tilde{\sigma}$ of $\HH\times\HH$. If $\tilde{\sigma}$ is in $\aut\HH\times\aut\HH$
it obviously normalizes $\Gamma$. Therefore, the factor preserving
automorphism group of $X_{\Gamma}$ is $N\Gamma/\Gamma$, where $N\Gamma$
is the normalizer of $\Gamma$ in $\aut\HH\times\aut\HH$. Altogether,
every automorphism is either represented by a coset $\gamma\Gamma$
with $\gamma\in N\Gamma$ or is of type $(\gamma\Gamma)\circ\mu$.
The following result is a key for our computations: 
\begin{thm}
\label{pro:Finite number fix pt} An automorphism $\sigma$ of $X_{\Gamma}$
has only finitely many fixed points or $\sigma$ is an involution
whose fixed point set is purely one-dimensional. The latter never
happens for a quaternionic fake quadric. \end{thm}
\begin{proof}
Suppose that $\gamma$ is in the subgroup $\aut\mathbb{H}\times\aut\mathbb{H}$.
Then, as explained above, $\sigma$ can be represented by a coset
$\gamma\Gamma$ with $\gamma\in N\Gamma$, where $N\Gamma$ is the
normalizer of $\Gamma$ in $G$. It is sufficient to show that $\gamma$
as a mapping $\mathbb{H}\times\mathbb{H}\longrightarrow\mathbb{H}\times\mathbb{H}$
has only finitely many fixed points in $\mathbb{H}\times\mathbb{H}$
modulo the action of $\Gamma$. Assume that $\gamma(z,w)=(\gamma^{\rho_{1}}z,\gamma^{\rho_{2}}w)=(z,w)$.
Then $\gamma$ is an elliptic transformation, i.e. $4\det(\gamma)-\text{Tr}(\gamma^{\rho_{i}})^{2}>0$
for $i=1,2$. The reason is the following: if $(z,w)$ is a fixed
point, then in particular $z$ is a fixed point of $\gamma^{\rho_{1}}$
and $w$ is a fixed point of $\gamma^{\rho_{2}}$. The only automorphisms
of $\mathbb{H}$ with fixed points in $\mathbb{H}$ are elliptic transformations.
Then, by definition $\gamma$ is elliptic.

Every non-trivial elliptic transformation of $\mathbb{H}$ has a unique
fixed point in $\mathbb{H}$ (the eigenvalue of the matrix which has
the positive imaginary part). Moreover, because $\Gamma$ is irreducible,
$\gamma^{\rho_{1}}$ is non-trivial if and only if $\gamma^{\rho_{2}}$
is non trivial (observe here that if $\Gamma$ were not irreducible,
there would be an automorphism of the form $(\gamma_{1},1)$ with non-isolated
fixed points).

Let thus $(z,w)$ be the unique fixed point of $\gamma$ in $\mathbb{H}\times\mathbb{H}$.
The $N\Gamma$-orbit of $(z,w)$ is discrete in $\mathbb{H}\times\mathbb{H}$,
therefore there is only one representative of $(z,w)$ modulo the
action of $N\text{\ensuremath{\Gamma}}$. Now, $N\Gamma$ is a finite
index extension of $\Gamma$, therefore there are only finitely many
representatives of $(z,w)$ modulo $\Gamma$.

Let us now suppose that the automorphism $\sigma$ is represented
by $\gamma\mu\in\aut(\mathbb{H}\times\mathbb{H})$. Then $(\gamma\mu)^{2}=(\gamma^{\rho_{1}}\gamma^{\rho_{2}},\gamma^{\rho_{2}}\gamma^{\rho_{1}})$
is an element of $\aut\mathbb{H}\times\aut\mathbb{H}$ that acts on
the surface. Suppose that $\sigma$ has an infinite number of fixed
points, then $\sigma^{2}$ must be the identity and $\gamma^{\rho_{1}}\gamma^{\rho_{2}}$
must be in $\Gamma$. A fixed point $(z,w)$ satisfies \[
(\gamma^{\rho_{1}}w,\gamma^{\rho_{2}}z)=\lambda(z,w)\]
for a $\lambda\in\Gamma$. After the change of $\gamma$ by $\lambda^{-1}\gamma$,
we can assume that $\lambda=1$, thus $w=\gamma^{\rho_{2}}z$ and
$\gamma^{\rho_{1}}\gamma^{\rho_{2}}=1$ because $\gamma^{\rho_{1}}\gamma^{\rho_{2}}$
is in $\Gamma$ that is torsion free. Reciprocally, let be $t\in\mathbb{H}$
; as $\gamma^{\rho_{1}}\gamma^{\rho_{2}}=1$ the point $(t,\gamma^{\rho_{2}}t)$
satisfy \[
\gamma\mu(t,\gamma^{\rho_{2}}t)=(t,\gamma^{\rho_{2}}t).\]
 Therefore there are no isolated fixed points for $\sigma$.

Assume now that $X_{\Gamma}$ is a quaternionic fake quadric. The
fixed locus $C$ of $\sigma$ is a smooth curve. The topological Lefschetz
formula (see Corollary \ref{cor:lefschetz top}) implies that the
genus of the irreducible components of $C$ is negative, thus the
automorphism has only a finite number of fixed points. 
\end{proof}

\section{Quaternionic fake quadrics with non-trivial automorphism groups.}

As already mentioned, a series of examples of quaternionic fake quadrics
has been constructed by I. Shavel in \cite{Shavel78}.
There, the author concentrates on arithmetic lattices $\Gamma\supseteq\Gamma_{\OO}^{1}$
which are defined by quaternion algebras over real quadratic fields
of class number one. More recently, in \cite{dz11}, more examples
of quaternionic quadrics associated with quaternion algebras over
quadratic fields have been found. In this section we will list all
known examples of quaternionic fake quadrics together with their automorphism
groups. We refer the reader to \cite{vign2} for generalities on quaternion
algebras.

Let us first make a few general observations, before we discuss the
examples in detail. For technical reasons it is more practical to
consider the group $\pgl_{2}^{+}(\RR)\times\pgl_{2}^{+}(\RR)$, where
$\pgl_{2}^{+}(\RR)=\Gl_{2}^{+}(\RR)/\RR^{\ast}$ and $\Gl_{2}^{+}(\RR)$
is the group of all $2\times2$ matrices with positive determinant,
instead of $\psl_{2}(\RR)\times\psl_{2}(\RR)$. We identify $\pgl_{2}^{+}(\RR)\times\pgl_{2}^{+}(\RR)$
with the group $\aut\HH\times\aut\HH$ of holomorphic automorphisms
which preserve the two factors. 

From the point of view of Theorem \ref{pro:Finite number fix pt} it is
more interesting to consider the automorphism subgroups $G\leq N\Gamma/\Gamma=:\aut(X_{\Gamma})$
of factor preserving automorphisms, which we will do in the following.
Since fake quadrics $X_{\Gamma}$ have relatively small covolume,
they tend to be large groups and therefore the order $|\aut(X_{\Gamma})|=|N\Gamma/\Gamma|$
is not too big. The normalizers $N\Gamma$ will be maximal lattices
and all such lattices can be described arithmetically as follows (see
\cite{borel}). \\
 If $X_{\Gamma}$ is a quaternionic fake quadric, there is an associated
tuple $(k,\rho_{1},\rho_{2},B,\OO)$ as described in Section \ref{sec:generalities}.
The quaternion algebra $B$ is for fixed $\rho_{1},\rho_{2}$ uniquely
determined (up to isomorphism) by the reduced discriminant $d_{B}=v_{1}\cdot\ldots\cdot v_{r}$,
the formal product over finite places $v_{i}$ of $k$ where $B$
is ramified, i.e. $B\otimes_{k}k_{v_{i}}\ncong M_{2}(k_{v_{i}})$,
hence $(k,\rho_{1},\rho_{2},B,\OO)=(k,\rho_{1},\rho_{2},d_{B},\OO)$.
In the following we will often abbreviate such a datum which determines
the quaternion algebra $B$ with $B(k,d_{B})$ or $B(k,v_{1}\ldots v_{r})$.
Let us fix such a datum $B(k,v_{1}\ldots v_{r})$ and let $B^{+}$
be the group of all $x\in B^{\ast}$ such that the reduced norm $\Nrd(x)$
is totally positive. It is known that \begin{equation}
N\Gamma_{\OO}^{+}=\{x\in B^{+}\mid x\OO x^{-1}=\OO\}/k^{\ast}\label{max}\end{equation}
 is a maximal lattice. 
$N\Gamma_{\OO}^{+}$ contains $\Gamma_{\OO}^{1}$ and it is known
that $\Gamma_{\OO}^{1}$ is normal in $N\Gamma_{\OO}^{+}$ with $N\Gamma_{\OO}^{+}/\Gamma_{\OO}^{1}\cong(\ZZ/2\ZZ)^{l}$
an elementary abelian $2$-group with $l\geq r$ and $r$ is the number
of ramified places in $B$ (see \cite{Shavel78} for instance). If
the class number of $k$ is one (as will be the case in all the consideres
examples) there is an alternative description of $N\Gamma_{\OO}^{+}$
as

\begin{align}
N & \Gamma_{\OO}^{+}=\{\alpha=\rho_{1}^{\epsilon_{1}}\cdots\rho_{r}^{\epsilon_{r}}\lambda\tau\in B^{\ast}\mid\label{max2}\\
 & \Nrd(\alpha)\ \text{totally positive},\ \tau\in k^{\ast},\ \lambda\in\OO^{\ast},\ \epsilon_{i}\in\{0,1\},\ \Nrd(\rho_{i})\ \text{divides}\ d_{B}\}/k^{\ast}\notag\end{align}

\noindent (see \cite{Shavel78}, p.~223). It follows that a quaternionic
fake quadric $X_{\Gamma}$ with $\Gamma\supseteq\Gamma_{\OO}^{1}$
will have an elementary abelian $2$-group as the automorphism group
$\aut(X_{\Gamma})$. All Shavel's examples will provide such automorphism
groups.

\subsection{A fake quadric with automorphism group \boldmath $\mathbb{Z}/2\mathbb{Z}$.}

There are examples of quaternionic fake quadrics $X_{\Gamma}$ whose
automorphism group is $\ZZ/2\ZZ$ and, as mentioned, they already
appear in \cite{Shavel78}.

For example, let $k=\QQ(\sqrt{2})$ and let $B=B(k,\pP_{3}\pP_{7})$
be the (unique) quaternion algebra over $k$ which is ramified exactly
at the two finite primes $\pP_{3}$ and $\pP_{7}$ of $k$ lying over
the rational primes $3$ and $7$ respectively. Since $k$ has the
class number one, there is the unique (up to conjugation) maximal
order $\OO$ in $B$. Consider the group $\Gamma_{\OO}^{1}$. By \cite{Shavel78},
Proposition 4.7, $X_{\Gamma_{\OO}^{1}}$ is smooth. By the already
mentioned general result of Matsushima and Shimura \cite{MatsushimaShimura},
$q(X_{\Gamma_{\OO}^{1}})=0$. The Euler number $c_{2}(X_{\Gamma_{\OO}^{1}})$
is computed via the volume formula of Shimizu (see \cite{Shavel78},
Theorem 3.1). Since the prime $3$ is inert and $7$ is decomposed
in $k$, this formula gives $c_{2}(X_{\Gamma_{\OO}^{1}})=8$. The
normalizer of $\Gamma_{\OO}^{1}$ is $N\Gamma_{\OO}^{+}$ and by \cite{Shavel78},
Proposition 1.3 and 1.4 we have \[
\aut(X_{\Gamma_{\OO}^{1}})\cong L_{1}/L_{2}=\langle[\pP_{3}],[\pP_{7}]\rangle\cong(\ZZ/2\ZZ)^{2},\]
 where $L_{1}$ is the group of principal fractional ideals of type
$(\pP_{3})(\pP_{7})I^{2}$ ($I$ a principal fractional ideal) for
which one can find a totally positive generator and $L_{2}$ consists
of all principal ideals of type $(a^{2})$ with $a\in k$ (See also
\cite{shim:zeta}, 3.12). Let $\Gamma_{\pP_{3}}$ be the kernel of
the canonical homomorphism \[
N\Gamma_{\OO}^{+}\longrightarrow L_{1}/L_{2}\longrightarrow\langle[\pP_{7}]\rangle.\]
 By Shavel's criterion (see \cite{Shavel78}, Theorem 4.11) $\Gamma_{\pP_{3}}$
is torsion free and as $[\Gamma_{\pP_{3}}:\Gamma_{\OO}^{1}]=2$, $X_{\Gamma_{\pP_{3}}}$
is a fake quadric with $\aut(X_{\Gamma_{\pP_{3}}})\cong\ZZ/2\ZZ$.

\subsection{A fake quadric with automorphism group \boldmath $(\mathbb{Z}/2\mathbb{Z})^{2}$.}

Consider again $k=\QQ(\sqrt{2})$ and now the quaternion algebra $B=B(k,\pP_{2}\pP_{5})$
over $k$ which is ramified exactly at the two finite places $\pP_{2}$
and $\pP_{5}$. Again there is the unique maximal order $\OO$ in
$B$ and as in the previous example, Shavel's results show that $X_{\Gamma_{\OO}^{1}}$
is smooth. The prime $2$ is ramified and $5$ is inert in $k$ and
therefore Shimizu's volume formula gives $c_{2}(X_{\Gamma_{\OO}^{1}})=4$.
Hence $X_{\Gamma_{\OO}^{1}}$ is a fake quadric. With the same arguments
as in the previous example $\aut(X_{\Gamma_{\OO}^{1}})$ is isomorphic
to $(\ZZ/2\ZZ)^{2}$.

\subsection{A fake quadric with automorphism group of order 20.}

\label{ord10} Consider $k=\QQ(\sqrt{5})$ and the quaternion algebra
$B=B(k,\pP_{2}\pP_{5})$ over $k$ which is ramified exactly at the
primes $\pP_{2}$ and $\pP_{5}$. In this case the group $\Gamma_{\OO}^{1}$
(where $\OO$ is again a maximal order in $B$), contains torsion
elements of order $5$ and no other torsions (see \cite{Shavel78},
Proposition 4.7 and Theorem 4.8) %
\footnote{ Note that the symbol $\left(\frac{}{p}\right)$ in Theorem 4.8 of
\cite{Shavel78} for $p=2$ should be read as the Kronecker symbol,
i.e. $\left(\frac{d}{2}\right)=1\Leftrightarrow d\equiv\pm1\bmod8$
and $=-1\Leftrightarrow d\equiv\pm3\bmod8$.%
}. Volume formula of Shimizu gives in this case $c_{2}(X_{\Gamma_{\OO}^{1}})=4/5$.
Let us now give a torsion-free subgroup $\Gamma<\Gamma_{\OO}^{1}$
of index $5$. The corresponding surface $X_{\Gamma}$ will be a fake
quadric. Since $\pP_{2}$ is ramified in $B$, there is a prime ideal
$\mathfrak{P}_{2}$ in $\OO$ lying over $\pP_{2}$ and satisfying
$\mathfrak{P}_{2}^{2}=\pP_{2}\OO$. Let \[
\Gamma=\Gamma_{\OO}^{1}(\mathfrak{P}_{2})=\{x\in\Gamma_{\OO}^{1}\mid x\equiv1\bmod\mathfrak{P}_{2}\}.\]
 $\Gamma_{\OO}^{1}(\mathfrak{P}_{2})$ is a normal subgroup in $\Gamma_{\OO}^{1}$
and the index can be computed via the localisation of $B$ at $\pP_{2}$.
Namely, observe first that $\Gamma_{\OO}^{1}/\Gamma_{\OO}^{1}(\mathfrak{P}_{2})$
is isomorphic to the factor group $\OO^{1}/\OO^{1}(\mathfrak{P}_{2})$,
where \[
\OO^{1}(\mathfrak{P}_{2})=\{x\in\OO^{1}\mid x\equiv1\bmod\mathfrak{P}_{2}\}.\]
 This is because $-1$ is in $\OO^{1}(\mathfrak{P}_{2})$. Let $\OO_{\pP_{2}}$
be the maximal order in $B_{\pP_{2}}$, i.e. $\OO_{\pP_{2}}=\OO\otimes_{\oo_{k}}\oo_{k_{\pP_{2}}}$,
where $\oo_{k_{\pP_{2}}}$ is the ring of integers in $k_{\pP_{2}}$.
Its maximal ideal $\widehat{\mathfrak{P}}_{2}$ is the topological
closure of $\mathfrak{P}_{2}$. By the strong approximation property
$\OO^{1}/\OO^{1}(\mathfrak{P}_{2})\cong\OO_{\pP_{2}}^{1}/\OO_{\pP_{2}}^{1}(\widehat{\mathfrak{P}}_{2})$
Note that $B_{\pP_{2}}^{1}=\OO_{\pP_{2}}^{1}$, since $\OO_{\pP_{2}}$
is the subring of $B_{\pP_{2}}$ consisting of elements whose reduced
norm is less or equal 1. We use a theorem of C. Riehm (see \cite{Riehm},
Theorem 7) by which \[
\OO_{\pP_{2}}^{1}/\OO_{\pP_{2}}^{1}(\widehat{\mathfrak{P}}_{2})\cong\ker((\OO_{\pP_{2}}/\widehat{\mathfrak{P}}_{2})^{\ast}\stackrel{Nr}{\longrightarrow}(\oo_{k_{\pP_{2}}}/\pP_{2})^{\ast})\cong\ker(\FF_{16}^{\ast}\stackrel{Nr}{\longrightarrow}\FF_{4}^{\ast})\cong\ZZ/5\ZZ\]
 (Note here that the norm map induces a surjective homomorphism of
multiplicative groups). Since $\Gamma_{\OO}^{1}(\mathfrak{P}_{2})$
is embedded in $\OO_{\pP_{2}}^{1}(\widehat{\mathfrak{P}}_{2})/\pm1$
and the latter group is a pro-2-group (again by \cite{Riehm}) it
can not contain elements of order $5$. Therefore, $\Gamma_{\OO}^{1}(\mathfrak{P}_{2})$
is a torsion-free group and $X_{\Gamma_{\OO}^{1}(\mathfrak{P}_{2})}$
is a fake quadric. Since $\Gamma_{\OO}^{1}$ contains a $5$-torsion
and $\Gamma_{\OO}^{1}$ normalizes $\Gamma_{\OO}^{1}(\mathfrak{P}_{2})$,
$X_{\Gamma_{\OO}^{1}(\mathfrak{P}_{2})}$ contains an automorphism
of order $5$. In order to determine the full automorphism group $\aut(X_{\Gamma_{\OO}^{1}(\mathfrak{P}_{2})})$
we first need to find the normaliser of $\Gamma_{\OO}^{1}(\mathfrak{P}_{2})$.
By definition elements of $N\Gamma_{\OO}^{+}$ normalise $\OO$, i.e.\ $x\OO x^{-1}=\OO$.
Let $\gamma\in\Gamma_{\OO}^{1}(\mathfrak{P}_{2})$. Since the class
number of $k$ is one, every two-sided $\OO$-ideal is principal and
we can choose $\Pi_{2}\in\OO$ such that $\Pi_{2}\OO=\mathfrak{P}_{2}$.
Moreover, as $\mathfrak{P}_{2}$ is uniquely determined by the property
that the $\oo_{k}$-ideal $\Nrd(\mathfrak{P}_{2})$ is $\pP_{2}$,
we can choose $\Pi_{2}$ such that $\Nrd(\Pi_{2})=2$. Then $\gamma=\pm(1+m\Pi_{2})$
with $m\in\OO$. For $x\in N\Gamma_{\OO}^{+}$ we have $x\gamma x^{-1}=1+xm\Pi_{2}x^{-1}=1+m'x\Pi_{2}x^{-1}$
with some $m'\in\OO$. The element $x\Pi_{2}x^{-1}$ lies in $\OO$
and $\Nrd(x\Pi_{2}x^{-1})=\Nrd(\Pi_{2})=2$. Since $\mathfrak{P}_{2}=\langle\Pi_{2}\rangle$
is the unique prime ideal over $2$, $x\Pi_{2}x^{-1}\in\mathfrak{P}_{2}$
and $x\gamma x^{-1}\in\Gamma_{\OO}^{1}(\mathfrak{P}_{2})$. It follows
that the normaliser of $\Gamma_{\OO}^{1}(\mathfrak{P}_{2})$ is $N\Gamma_{\OO}^{+}$.
This leads to an exact sequence 

\begin{equation}
1\longrightarrow\Gamma_{\OO}^{1}/\Gamma_{\OO}^{1}(\mathfrak{P}_{2})\longrightarrow N\Gamma_{\OO}^{+}/\Gamma_{\OO}^{1}(\mathfrak{P}_{2})\longrightarrow N\Gamma_{\OO}^{+}/\Gamma_{\OO}^{1}\longrightarrow1\label{exact}\end{equation}
 which we can write abstractly as \[
1\longrightarrow\mathbb{Z}/5\mathbb{Z}\longrightarrow\aut(X_{\Gamma_{\OO}^{1}(\Pi_{2})})\longrightarrow\mathbb{Z}/2\mathbb{Z}\times\mathbb{Z}/2\mathbb{Z}\longrightarrow1.\]
Let $\lambda\in\OO^{1}$ satisfy $\lambda^{5}=-1$, i.~e.~$\lambda$
gives rise to a $5$-torsion in $\Gamma_{\OO}^{1}$. Then $\lambda$
satisfies the equation $\lambda^{2}-\frac{1+\sqrt{5}}{2}\lambda+1=0$
over $k$. We can assume that $\lambda$ generates $\Gamma_{\OO}^{1}/\Gamma_{\OO}^{1}(\mathfrak{P}_{2})$.
Let $g=\lambda+1$. The reduced norm of $g$ is $\Nrd(g)=(\lambda+1)(\overline{\lambda}+1)=\Nrd(\lambda)+\Trd(\lambda)+1=2+\frac{1+\sqrt{5}}{2}=\frac{5+\sqrt{5}}{2}$,
where $\Trd$ is the reduced trace. Since $\frac{5+\sqrt{5}}{2}$
is a totally positive generator of the prime ideal over $5$, $g$
defines an element of $N\Gamma_{\OO}^{+}$ (see (\ref{max2})). On
the other hand $g^{2}=(\lambda+1)^{2}=\lambda^{2}+2\lambda+1=(\frac{1+\sqrt{5}}{2}\lambda-1)+2\lambda+1=(\frac{5+\sqrt{5}}{2})\lambda$.
This shows that $g$ has order $10$ in $N\Gamma_{\OO}^{+}$ and hence
gives an element of order $10$ in $N\Gamma_{\OO}^{+}/\Gamma_{\OO}^{1}(\mathfrak{P}_{2})$.
Moreover the image of $g$ in $N\Gamma_{\OO}^{+}/\Gamma_{\OO}^{1}$
is not trivial. Using the computer algebra system PARI, we can check that both ramified primes
$\pP_{2}$ and $\pP_{5}$ are not split in $k(\sqrt{-2})$. This implies
that $k(\sqrt{-2})\subset B$ (see \cite{Shavel78}, Proposition 4.5)
and we can take $\sqrt{-2}$ as the generator $\Pi_{2}$ of $\mathfrak{P}_{2}$.
Hence, $\Pi_{2}$, considered as an element of $N\Gamma_{\OO}^{+}$, is of order
$2$ and the images of $g$ and $\Pi_{2}$ in $N\Gamma_{\OO}^{+}/\Gamma_{\OO}^{1}$
generate this group. 

\begin{lem}
\label{commutativity_of_g_and_h} Let $g$ and $\Pi_{2}$ be elements
constructed above Then in $N\Gamma_{\OO}^{+}$ we have the relation
$\Pi_{2} g\Pi_{2}=g^{-1}$ modulo $\Gamma_{\OO}^{1}(\mathfrak{P}_{2})$ \end{lem}
\begin{proof}
The element $\Pi_{2}$ generates $\mathfrak{P}_{2}$. Consider $g$
and $\Pi_{2}$ as the elements of the localisation $B_{\pP_{2}}$
of $B$ at $\pP_{2}$. This is a division quaternion algebra over
$k_{\pP_{2}}$ and has a representation \[
B_{\pP_{2}}=L_{\pP_{2}}\oplus\Pi_{2}L_{\pP_{2}},\]
 where $L_{\pP_{2}}$ is the unique unramified quadratic extension
of $k_{\pP_{2}}$ (see \cite{vign2}, p.34). For every $t\in L_{\pP_{2}}$
we have $t\Pi_{2}=\Pi_{2}\overline{t}$, where $\overline{t}$ is
the Galois-conjugate of $t$ in $L_{\pP_{2}}$. The element $g$ lies
in $k_{\pP_{2}}(\lambda)=k_{\pP_{2}}(\xi_{5})$ which is an unramified
quadratic extension of $k_{\pP_{2}}$, so $L_{\pP_{2}}=k_{\pP_{2}}(\lambda)$.
Therefore $g\in L_{\pP_{2}}$ and $g\Pi_{2}=\Pi_{2}\overline{g}$.
Because $g\overline{g}$ is in $k^{\ast}$ we have that $\overline{g}=g^{-1}$
considered as an element of $N\Gamma_{\OO}^{+}\subset B_{\pP_{2}}^{\ast}/k_{\pP_{2}}^{\ast}$.
This gives a relation $\Pi_{2}g\Pi_{2}=g^{-1}$ in $N\Gamma_{\OO}^{+}$, since $\Pi_2^2=1$ 
in $N\Gamma_{\OO}^+$. Also, $g$ and $\Pi_{2} g\Pi_{2}=g^{-1}$ are not equal modulo
$\Gamma_{\OO}^{1}(\mathfrak{P}_{2})$ because this would imply that
$g^{2}\in\Gamma_{\OO}^{1}(\mathfrak{P}_{2})$. But as $\Gamma_{\OO}^{1}(\mathfrak{P}_{2})$
is torsion-free and $g^{2}$ is of finite order, this is impossible. \end{proof}
\begin{prop}
With above notations we have \[
\aut(X_{\Gamma_{\OO}^{1}(\mathfrak{P}_{2})})\cong\mathbb{D}_{10}.\]
 \end{prop}
\begin{proof}
By the above discussion, $\aut(X_{\Gamma_{\OO}^{1}(\mathfrak{P}_{2})})$
is of order $20$ and is generated by elements $g$ of order $10$
and $\Pi_{2}$ of order $2$ satisfyng $\Pi_{2} g\Pi_{2}=g^{-1}$.
The only group of order 20 with these relations is $\mathbb{D}_{10}$. 
\end{proof}

\subsection{A fake quadric with automorphism group of order 8.}

We consider $k=\QQ(\sqrt{5})$ and $B=B(k,\pP_{2}\pP_{11})$, the
unique quaternion algebra ramified exactly at the primes $\pP_{2}$
and $\pP_{11}$. Since $2$ is inert and $11$ is decomposed in $k$,
Shimizu's volume formula gives $c_{2}(X_{\Gamma_{\OO}^{1}})=\frac{4}{5\cdot12}(4-1)(11-1)=2$
as the value of the second Chern number of the quotient $X_{\Gamma_{\OO}^{1}}$,
where again $\Gamma_{\OO}^{1}$ is the norm-1 group of a maximal order
in $B$. As before, results of \cite{Shavel78} show that $\Gamma_{\OO}^{1}$
contains only torsion elements of order $2$ and no other torsions
(Here, observe that $2$ is split in $\QQ(\sqrt{-15})$, hence, by \cite{Shavel78}, 
Theorem 4.8 there are no elements of order $3$ in $\Gamma_{\OO}^{1}$, and note that
there are no elements of order $5$ because $11\equiv1\bmod5$ which
implies that $\pP_{11}$ is split in $k(\xi_{5})$ ). Since $\pP_{11}$
is ramified in $B$, there is the unique prime ideal $\mathfrak{P}_{11}$
in $\OO$ such that $\mathfrak{P}_{11}^{2}=\pP_{11}\OO$. Consider
the principal congruence subgroup \[
\Gamma_{\OO}^{1}(\mathfrak{P}_{11})=\{x\in\Gamma_{\OO}^{1}\mid x\equiv1\bmod\mathfrak{P}_{11}\}.\]
 It is a normal subgroup in $\Gamma_{\OO}^{1}$. The quotient $\Gamma_{\OO}^{1}/\Gamma_{\OO}^{1}(\mathfrak{P}_{11})$
is isomorphic to $\OO^{1}/\pm\OO^{1}(\mathfrak{P}_{11})$ because
$-1\notin\OO^{1}(\mathfrak{P}_{11})$. In order to compute the latter
quotient we change over to the localisation at the prime $\pP_{11}$.
Let \[
B_{\pP_{11}}=B\otimes_{k}k_{\pP_{11}}=B\otimes_{k}\QQ_{11}.\]
 It is the unique division quaternion algebra over $\QQ_{11}$. Let
us write $\OO_{\pP_{11}}=\OO\otimes_{\oo_{k}}\ZZ_{11}$ for its maximal
order. As in the previous example let $\widehat{\mathfrak{P}}_{11}$
denote the prime ideal of $\OO_{\pP_{11}}$. We have \[
\OO^{1}/\OO^{1}(\mathfrak{P}_{11})\cong\OO_{\pP_{11}}^{1}/\OO_{\pP_{11}}^{1}(\widehat{\mathfrak{P}}_{11})\]
 by the strong approximation theorem. By C. Riehm's result, \cite{Riehm},
Theorem 7, \[
\OO_{\pP_{11}}^{1}/\OO_{\pP_{11}}^{1}(\widehat{\mathfrak{P}}_{11})\cong\ker((\OO_{\pP_{11}}/\widehat{\mathfrak{P}}_{11})^{\ast}\stackrel{Nr}{\longrightarrow}(\oo_{k_{\pP_{11}}}/\pP_{11})^{\ast})\cong\ker(\FF_{121}^{\ast}\longrightarrow\FF_{11}^{\ast}).\]
 Since $\FF_{121}=\FF_{11}(\xi_{12})$, where $\xi_{12}$ denotes
a primitive twelfth root of unity we conclude that $\OO_{\pP_{11}}^{1}/\OO_{\pP_{11}}^{1}(\widehat{\mathfrak{P}}_{11})$
is isomorphic to $\mu_{12}=\langle\xi_{12}\rangle$. Hence \[
\Gamma_{\OO}^{1}/\Gamma_{\OO}^{1}(\mathfrak{P}_{11})\cong\OO_{\pP_{11}}^{1}/\pm\OO_{\pP_{11}}^{1}(\widehat{\mathfrak{P}}_{11})\cong\mu_{6}=\langle\xi_{6}\rangle.\]
 Let us now define an intermediate group \[
\Gamma=\{x\in\Gamma_{\OO}^{1}\mid x\bmod\mathfrak{P}_{11}\in\langle\xi_{6}^{2}\rangle\subset\mu_{6}\}.\]
 $\Gamma<\Gamma_{\OO}^{1}$ is a subgroup of index $2$, hence $c_{2}(X_{\Gamma})=4$.
Moreover, $\Gamma$ is torsion-free since it can not contain elements
of order $2$. For if an order-two element $x$ is in $\Gamma$, then
its image $x\bmod\mathfrak{P}_{11}$ in $\Gamma_{\OO}^{1}/\Gamma_{\OO}^{1}(\mathfrak{P}_{11})$
lies in a cyclic group $\langle\xi_{6}^{2}\rangle$ of order three,
hence it must be the identity. But this means that $x$ is in $\Gamma_{\OO}^{1}(\mathfrak{P}_{11})$.
On the other hand $\Gamma_{\OO}^{1}(\mathfrak{P}_{11})$ is torsion-free
because it embeds in a pro-11 group $\OO_{\pP_{11}}^{1}(\widehat{\mathfrak{P}}_{11})/\pm1$.
This contradicts the assumption on $x$. All this shows that $X_{\Gamma}$
is a fake quadric. 
\begin{prop}
Let $N\Gamma_{\OO}^{+}$ be defined as in (\ref{max}). Then $N\Gamma_{\OO}^{+}$
is the normaliser of $\Gamma$ and $N\Gamma_{\OO}^{+}/\Gamma$ is
isomorphic to $\mathbb{D}_{4}$. \end{prop}
\begin{proof}
As a subgroup of index two in $\Gamma_{\OO}^{1}$ the group $\Gamma$
is normal in $\Gamma_{\OO}^{1}$. On the other hand, for the same
reason as in the previous example, $\Gamma_{\OO}^{1}(\mathfrak{P}_{11})$
as well as $\Gamma_{\OO}^{1}$ is normal subgroup in $N\Gamma_{\OO}^{+}$.
This already implies that $\Gamma$ is normal in $N\Gamma_{\OO}^{+}$
because any conjugate of $\Gamma$ will be a subgroup between $\Gamma_{\OO}^{1}(\mathfrak{P}_{11})$
and $\Gamma_{\OO}^{1}$ of index $2$ in $\Gamma_{\OO}^{1}$. There
is only one such group, namely $\Gamma$, since $\Gamma_{\OO}^{1}/\Gamma_{\OO}^{1}(\mathfrak{P}_{11})\cong\ZZ/6\ZZ$.
Similar exact sequence as (\ref{exact}) now shows that $\aut(X_{\Gamma})$
is an extension of $\ZZ/2\ZZ$ by the Klein's four group. Since the 2-torsions in $\Gamma_{\OO}^1$ come from 
embeddings of fourth root of unity into $\OO$ there is $\lambda\in\OO^{1}$ such that $\lambda^{2}=-1$. Let
$g=\lambda+1$. Then, as $\Trd(\lambda)=0$, we have $\Nrd(g)=(\lambda+1)(\overline{\lambda}+1)=2$
and also $g^{2}=(\lambda+1)^{2}=2\lambda$ which implies that $g$
defines an element of order $4$ in $N\Gamma_{\OO}^{+}$ and hence
an element of order $4$ in $N\Gamma_{\OO}^{+}/\Gamma$. Moreover,
the image of $g$ in $N\Gamma_{\OO}^{+}/\Gamma_{\OO}^{1}$ is not
trivial. Since both prime divisors $2$ and $\pi_{11}$ of the reduced
discriminant do not split in $k(\sqrt{-\pi_{11}})$ (as can be checked
using PARI for instance), the element $\Pi_{11}=\sqrt{-\pi_{11}}$
is in $B$ and moreover $\Pi_{11}$ defines an element of $N\Gamma_{\OO}^{+}$
of order $2$ such that the images of $\Pi_{11}$ and $g$ in $N\Gamma_{\OO}^{+}/\Gamma_{\OO}^{1}$
generate this group. 
Same argument as in Lemma \ref{commutativity_of_g_and_h} gives a
relation between $\Pi_{11}$ and $g$: Consider $\Pi_{11}$ as the
generator of the prime ideal $\mathfrak{P}_{11}$. Locally, $B_{\pP_{11}}$
can be written as $B_{\pP_{11}}=L_{\pP_{11}}\oplus\Pi_{11}L_{\pP_{11}}$,
where $L_{\pP_{11}}=k_{\pP_{11}}(\xi_{12})$ is the unique unramified
quadratic extension of $k_{\pP_{11}}\cong\QQ_{11}$ with the multiplication
rule $t\Pi_{11}=\Pi_{11}\overline{t}$ for all $t\in L_{\pP_{11}}$.
The element $g$ is in $L_{\pP_{11}}$, namely $g=1+\xi_{12}^{3}$.
Then $g\Pi_{11}=\Pi_{11}\overline{g}=h(1+\overline{\xi_{12}}^{3})=(1+\xi_{12}^{9})$.
In $N\Gamma_{\OO}^{+}$ the relations $\overline{g}=g^{-1}$ and $\Pi_{11}^2=1$ hold, 
hence $\Pi_{11}g\Pi_{11}=g^{-1}$ in $N\Gamma_{\OO}^{+}$. Also $g\neq g^{-1}$ modulo $\Gamma$, since
otherwise $g^{2}$ would be in $\Gamma$ which is not possible because
$g^{2}$ is torsion and $\Gamma$ torsion-free. $N\Gamma_{\OO}^{+}/\Gamma$
is isomorphic to $\mathbb{D}_{4}$ which is the only group of order
$8$ generated by two elements $\Pi_{11}$ of order $2$ and $g$
of order $4$ with $h\neq g^{2}$ and $\Pi_{11} g\Pi_{11}=g^{-1}$. \end{proof}
\begin{rem}
Considering $k=\mathbb{Q}(\sqrt{13})$, the quaternion algebra $B=B(k,\pP_{2}\pP_{3})$
and $\Gamma=\Gamma_{\OO}^{1}(\mathfrak{P}_{3})$, the arguments as
in the examples before will show that $X_{\Gamma_{\OO}^{1}(\mathfrak{P}_{3})}$
is a fake quadric whose automorphism group is isomorphic to $\mathbb{D}_{4}$. 
\end{rem}

\subsection{A fake quadric with an automorphism group \boldmath $\mathbb{D}_{6}$.}

This time we consider the quadratic field $k=\QQ(\sqrt{2})$ and the
quaternion algebra $B=B(k,\pP_{2}\pP_{3})$. The norm-1 group $\Gamma_{\OO}^{1}$
of a maximal order in $B$ contains torsion elements of order $3$,
but no elements of order $2$, because $\pP_{3}$ is decomposed in
$k(\sqrt{-1})$. The second Chern number of the quotient $X_{\Gamma_{\OO}^{1}}$
is $c_{2}(X_{\Gamma_{\OO}^{1}})=(9-1)/6=4/3$. Let $\Gamma_{\OO}^{1}(\mathfrak{P}_{2})$
be the principal congruence subgroup corresponding to the prime ideal
$\mathfrak{P}_{2}\subset\OO$, defined by the relation $\mathfrak{P}_{2}^{2}=\pP_{2}\OO$.
Again by Riehm's theorem and with arguments as in Section \ref{ord10},
$\Gamma_{\OO}^{1}(\mathfrak{P}_{2})$ is torsion free normal subgroup
in $\Gamma_{\OO}^{1}$ of index $3$, hence $X_{\Gamma_{\OO}^{1}(\mathfrak{P}_{2})}$
is a fake quadric. The automorphism group $\aut(X_{\Gamma_{\OO}^{1}(\mathfrak{P}_{2})})$
is isomorphic to the factor group \[
N\Gamma_{\OO}^{+}/\Gamma_{\OO}^{1}(\mathfrak{P}_{2}).\]
 which is an extension of $\Gamma_{\OO}^{1}/\Gamma_{\OO}^{1}(\mathfrak{P}_{2})\cong\ZZ/3\ZZ$
by $N\Gamma_{\OO}^{+}/\Gamma_{\OO}^{1}\cong\ZZ/2\ZZ\times\ZZ/2\ZZ$.
\begin{prop}
We have $\aut(X_{\Gamma_{\OO}^{1}(\mathfrak{P}_{2})})\cong\mathbb{D}_{6}$. \end{prop}
\begin{proof}
Let $\lambda\in\OO^{1}$ be an element with $\lambda^{3}=-1$ and
$g=\lambda+1$. Such $\lambda$ exists since $\Gamma_{\OO}^{1}$ contains
$3$-torsions. We can take $\pm\lambda$ to be the generator of $\Gamma_{\OO}^{1}/\Gamma_{\OO}^{1}(\mathfrak{P}_{2})$.
Since $\Trd(\lambda)=1$, we have $\Nrd(g)=3$ which implies that
$g$ defines an element in $N\Gamma_{\OO}^{+}$. Additionally $g^{2}=\lambda^{2}+2\lambda+1=3\lambda$
which means that $g$ has order $6$ considered as an element of $N\Gamma_{\OO}^{+}$.
The totally positive element $\pi_{2}=2+\sqrt{2}\in k$ generates
$\pP_{2}$ and since neither $\pi_{3}=3$ nor $\pi_{2}$ are split
in $k(\sqrt{-\pi_{2}})$, $\Pi_{2}=\sqrt{-\pi_{2}}$ lies in $B$
and defines an element in $N\Gamma_{\OO}^{+}$ of order $2$ such
that the classes of $g$ and $\Pi_{2}$ in $N\Gamma_{\OO}^{+}/\Gamma_{\OO}^{1}$
generate this group. In particular, $\Pi_{2}$ is a generator of $\mathfrak{P}_{2}$.
Locally $B_{\pP_{2}}=L_{\pP_{2}}\oplus\Pi_{2}L_{\pP_{2}}$, where
$L_{\pP_{2}}=\QQ_{2}(\xi_{6})$ is the unramified quadratic extension
of $k_{\pP_{2}}\cong\QQ_{2}$. As in previous examples, $g$ lies
in $L_{\pP_{2}}$ and $\Pi_{2}g\Pi_{2}=\overline{g}=g^{-1}$
in $N\Gamma_{\OO}^{+}$. This gives a relation $\Pi_{2} g\Pi_{2}=h^{-1}$
in $N\Gamma_{\OO}^{+}/\Gamma_{\OO}^{1}(\mathfrak{P}_{2})$. As $\Pi_{2}$
is not a power of $g$, the finite group generated by $g$ and $\Pi_{2}$
is isomorphic to $\mathbb{D}_{6}$.
\end{proof}

\subsection{Automorphism groups of order 16 and 24}

\label{other} There are more examples of quaternionic fake quadrics
with a non-trivial automorphism group. For instance, all examples
in Shavel's paper have $\ZZ/2\ZZ$ or $(\ZZ/2\ZZ)^{2}$ as the full
group of automorphisms. As in previous examples we show 
\begin{prop}
Let $B(\QQ(\sqrt{2}),\pP_{2},\pP_{7})$ be the indefinite quaternion
algebra over $k=\QQ(\sqrt{2})$ with reduced discriminant $d_{B}=\pP_{2}\pP_{7}$
and $\Gamma_{\OO}^{1}(\mathfrak{P}_{7})$ the congruence subgroup
in $\Gamma_{\OO}^{1}$ corresponding to a maximal order $\OO$ in
$B$ with respect to the prime ideal $\mathfrak{P}_{7}$ of $\OO$
lying over the ramified prime $\pP_{7}$. Then $X_{\Gamma_{\OO}^{1}(\mathfrak{P}_{7})}$
is a fake quadric with the automorphism group $\aut(X_{\Gamma_{\OO}^{1}(\mathfrak{P}_{7})})\cong\mathbb{D}_{8}$.\end{prop}
\begin{proof}
The proof goes along the same lines as in the examples before. By
Riehm's Theorem, $\Gamma_{\OO}^{1}/\Gamma_{\OO}^{1}(\mathfrak{P}_{7})\cong\ZZ/4\ZZ$
and we obtain $c_{2}(X_{\Gamma_{\OO}^{1}(\mathfrak{P}_{7})})=4$ by
Shimizu's formula. By Shavel's criterion for the existence of torsions,
we find that the maximal order $\OO$ contains a primitive eighth
root of unity $\lambda$ which leads to an element of order $4$ in
$\Gamma_{\OO}^{1}$. We can take $\lambda$ as a generator of this
quotient. As in the examples before take $g=1+\lambda$. Then, as
$\lambda$ satisfies $\lambda^{2}-\sqrt{2}\lambda+1=0$ over $k$,
$\Nrd(g)=\Nrd(\lambda+1)=2+\sqrt{2}$, hence $g$ defines an element
in $N\Gamma_{\OO}^{+}$. We have $g^{2}=\lambda^{2}+2\lambda+1=\sqrt{2}\lambda+2\lambda=(2+\sqrt{2})\lambda$.
Hence, $g$ is an element of order $8$ in $N\Gamma_{\OO}^{+}$ and
its image in $N\Gamma_{\OO}^{+}/\Gamma_{\OO}^{1}$ is not trivial.
The rational prime $7$ is split in $k$, hence, there are two possible
choices of $\pP_{7}$. Fix a prime $\pP_{7}=\langle\pi_{7}\rangle$
($\pi_{7}=3+\sqrt{2}$ say). Both $\pi_{7}$ as well as $\pi_{2}$
are ramified in $k(\sqrt{-\pi_{7}})$, hence $\sqrt{-\pi_{7}}\in B$
defines an element $\Pi_{7}\in B$ which defines an order-$2$ element
in $N\Gamma_{\OO}^{+}$. As in the previous examples we have $\Pi_{7} g\Pi_{7}=\overline{g}$
because locally in $B_{\pP_{7}}$, $\Pi_{7}=\sqrt{-\pi_{7}}$ generates
the unique prime ideal of the maximal order $\mathcal{O}_{\pP_{7}}$
and $g$ lies in the unramified quadratic extension $L_{\pP_{7}}=\QQ_{7}(\xi_{8})$.
This gives a relation $\Pi_7g\Pi_7=g^{-1}$ in $N\Gamma_{\OO}^{+}/\Gamma_{\OO}^{1}(\mathfrak{P}_{7})$.
Also $\Pi_{7}$ is not a power of $g$ modulo $\Gamma_{\OO}^{1}(\mathfrak{P}_{7})$
since the reduced norms of $\Pi_{7}$ and $g$ are different primes.
The only group of order $16$ with these relations is $\mathbb{D}_{8}$.
\end{proof}
Let us finally sketch the construction of a fake quadric with an automorphism
group of order $24$. 
\begin{prop}
Let $B(\QQ(\sqrt{3}),\pP_{2},\pP_{3})$ be the indefinite quaternion
algebra over $k=\QQ(\sqrt{3})$ ramified over the prime ideals $\pP_{2}$
and $\pP_{3}$ and let $\Gamma_{\OO}^{1}(\mathfrak{P}_{2}\mathfrak{P}_{3})\triangleleft\Gamma_{\OO}^{1}$
be the principal congruence subgroup with respect to the principal
ideal $\mathfrak{P}_{2}\mathfrak{P}_{3}$ of a maximal order $\OO\subset B$
lying over $\pP_{2}\pP_{3}$. Then $X_{\Gamma_{\OO}^{1}(\mathfrak{P}_{2}\mathfrak{P}_{3})}$
is a fake quadric with $|\aut(X_{\Gamma_{\OO}^{1}(\mathfrak{P}_{2}\mathfrak{P}_{3})})|=24$.
$\aut(X_{\Gamma_{\OO}^{1}(\mathfrak{P}_{2}\mathfrak{P}_{3})})$ contains
a cyclic subgroup of order $12$ \end{prop}
\begin{rem}
The full automorphism group in this case has order $24$. To our knowledge,
this is the largest known automorphism group of a fake quadric. The precise abstract group structure 
of $\aut(X_{\Gamma_{\OO}^{1}(\mathfrak{P}_{2}\mathfrak{P}_{3})})$ is not known to us,
since the local method, used in previous examples does not apply directly in this case. \end{rem}
\begin{proof}
That $X_{\Gamma_{\OO}^{1}(\mathfrak{P}_{2}\mathfrak{P}_{3})}$ has
the correct numerical invariants follows again from Riehm's Theorem, Shimizu's formula
and the observation that for the index we have 
$[\Gamma_{\OO}^1:\Gamma_{\OO}^1(\mathfrak{P}_2\mathfrak{P}_3)]=[\Gamma_{\OO}^1:\Gamma_{\OO}^1(\mathfrak{P}_2)][\Gamma_{\OO}^1:\Gamma_{\OO}^1(\mathfrak{P}_3)]$ 
. By Shavel's criterion, $B$ contains $k(\xi_{12})$
where $\xi_{12}$ is a primitive twelfth root of unity, hence there
is an element $\lambda\in\OO$ with $\lambda^{6}=-1$. So to show
that $\Gamma_{\OO}^{1}(\mathfrak{P}_{2}\mathfrak{P}_{3})$ is torsion
free we have to exclude the existence of 6-torsions in $\Gamma_{\OO}^{1}(\mathfrak{P}_{2}\mathfrak{P}_{3})$.
But since the reduced trace of $\lambda$ is $\pm\sqrt{3}$ which
is not congruent $2$ modulo $\pP_{2}\pP_{3}$, $\lambda$ is not
contained in $\Gamma_{\OO}^{1}(\mathfrak{P}_{2}\mathfrak{P}_{3})$.
The element $g=\lambda+1$ has norm $\Nrd(g)=2+\sqrt{3}$ which is
a totally positive unit of $\oo_{k}$ unit, hence $g$ lies in $\Gamma_{\OO}^{+}=\OO^{+}/\oo_{k}^{\ast}$,
where $\OO^{+}$ denotes the group of all units whose reduced norm
is totally positive. The group $\Gamma_{\OO}^{+}$ which is an index-2-extension
of $\Gamma_{\OO}^{1}$ since the fundamental unit $2+\sqrt{3}$ is
totally positive. Also $g^{2}=(2+\sqrt{3})\lambda$ which shows that
$g$ has order $12$ in $\Gamma_{\OO}^{+}\triangleleft N\Gamma_{\OO}^{+}$.
The image of $g$ in $N\Gamma_{\OO}^{+}/\Gamma_{\OO}^{1}$ is not
trivial and the discussion in \cite{Shavel78}, p.~223-224 shows
that $N\Gamma_{\OO}^{+}/\Gamma_{\OO}^{\ast}$ is generated by the
class of an element $\Pi\in N\Gamma_{\OO}^{+}$ with $\Nrd(\Pi)=\pi_{2}\pi_{3}$
where $\pP_{2}=\langle\pi_{2}\rangle,\pP_{3}=\langle\pi_{3}\rangle$
(note that the generators $\pi_{2}$ and $\pi_{3}$ can not be chosen
totally positive). Therefore, $Aut(X_{\Gamma_{\OO}^1((\mathfrak{P}_{2}\mathfrak{P}_{3}))})$ is 
of order $24$ and is an extension of $\ZZ/6\ZZ$ by the Klein's four group. 
\end{proof}

\section{Computations of the quotient surfaces}

Let $S$ be a quaternionic fake quadric, $G$ a group of automorphisms
of $S$, $X=S/G$ the quotient surface and let $\pi:Z\to S/G$ be
the minimal desingularisation map. 

Let us first study the case where $G$ is generated by an involution
$\sigma$. 
\begin{prop}
\label{pro:The-involution-} An involution $\sigma$ has $4$ fixed
points. The invariants of $Z$ are :\[
K_{Z}^{2}=4,\, c_{2}=8,\, q=p_{g}=0,\, h^{1,1}=5.\]
 The surface $Z$ is minimal of general type.\end{prop}
\begin{proof}
By Lefschetz formula (Proposition \ref{pro:Lefschetz formula}), $1=\sum_{s=\sigma(s)}\frac{1}{4}$,
therefore $\sigma$ has $4$ fixed points. Their images in $S/\sigma$
are $4$ $A_{1}$ singularities, resolved by $4$ $(-2)$-curves on
$Z$. The invariants of $Z$ are easy to compute.\\
The surface $Z$ is of general type and is minimal because $K_{Z}$
is the pullback of the nef divisor $K_{X}$.\end{proof}
\begin{prop}
\label{pro:Quotient by 3}Let be $G=\mathbb{Z}/3\mathbb{Z}$. The
singularities of quotient surface $X$ are $2A_{3,1}+2A_{3,2}$. The
surface $Z$ has general type and: \[
K_{Z}^{2}=2,\,\text{\ensuremath{c_{2}}=10},\, q=p_{g}=0.\]
\end{prop}
\begin{proof}
We use the notations of Zhang's formula (Proposition \ref{prop Zhang formula}).
In this case this formula gives $r_{1}+r_{2}=4$. A $A_{3,1}$ singularity
is resolved by a $(-3)$-curve and we have \[
K_{Z}^{2}=\frac{8}{3}-\frac{r_{1}}{3}.\]
 Therefore $r_{1}=2$ and $r_{2}=2$. The singularities of $X=S/\sigma$
are $2A_{3,1}+2A_{3,2}$. Moreover, as $q=p_{g}=0$, we have $c_{2}=10$.
$Z$ is of general type by Lemma \ref{lem: when has general type}.\end{proof}
\begin{prop}
\label{pro:Quotient by 3 times 3}There is no quaternionic fake quadric
with $G=(\mathbb{Z}/3\mathbb{Z})^{2}\subset\aut X$.\end{prop}
\begin{proof}
Let $\sigma_{1},\sigma_{2}$ be the two generating elements of $G$.
Let $p$ be a fixed point of $\sigma_{1}$. Then as $\sigma_{1}$
and $\sigma_{2}$ commute, $\sigma_{2}p,\sigma_{2}^{2}p$ are also
fixed points of $\sigma_{1}$ and there are one or four points fixed
by the whole group. Let $p$ be such a fixed point. Then there must
be a faithful action of $G$ on $T{}_{S,p}$. But such a faithful
action has elements which have eigenvalue $1$. This is impossible
as an automorphism has only isolated fixed points.\end{proof}
\begin{prop}
\label{prop order 4}Let be $G=\mathbb{Z}/4\mathbb{Z}$. The singularities
of the quotient $X$ are $2A_{4,1}+2A_{4,3}$ or $A_{1}+2A_{4,3}$.
The invariants of the resolution $Z$ are \[
K_{Z}^{2}=0,\, c_{2}=12,\, q=p_{g}=0\]
 in the first case, and in the second case $Z$ is minimal and satisfies
\[
K_{Z}^{2}=2,c_{2}=10,q=p_{g}=0.\]

\end{prop}
\begin{rem}
Proposition \ref{pro:Let-be-Dihedral D4} gives an example of the
first case. 
\end{rem}
\begin{proof}
Let $s$ be a fixed point of an order $4$ automorphism $\sigma$
acting on $S$. As the involution $\sigma^{2}$ has only isolated
fixed points, the eigenvalues of $\sigma$ acting on $T_{S,s}$ cannot
be $(i,-1)$ or $(-i,-1)$. Let $a,b,c$ be the number of fixed points
such that the eigenvalues of $\sigma$ are $(i,i),\,(-i,-i)$ and
$(i,-i)$ respectively. The Lefschetz holomorphic fixed point formula
implies\[
-\frac{a}{2i}+\frac{b}{2i}+\frac{c}{2}=1\ \text{and}\ a+b+c=4\ \text{or}\ 2,\]
 thus there are two cases : \\
 1) $a=b=1$ and $c=2$. The singularities of $S/G$ are $2A_{4,1}+2A_{4,3}$.
\\
 2) $a=b=0$ and $c=2$. In this case, the singularities of $S/G$
are $A_{1}+2A_{4,3}$ because $\sigma^{2}$ has $4$ fixed points.

A $A_{4,1}$ singularity is resolved by a $(-4)$-curve $C_{k}$.
A $A_{4,3}$ singularity is resolved by a chain of three $(-2)$-curve
and we have\[
K_{Z}=\pi^{*}K_{S/\sigma}-\sum_{k=1}^{k=2}\frac{1}{2}C_{k},\]
 thus $K_{Z}^{2}=\frac{8}{4}-2=0$ in the first case. Additionally,\[
e(S/\sigma)=\frac{1}{4}(4+(4-1)4)=4,\]
 thus $c_{2}(Z)=4+8=12$. The invariants in the second case are computed
in a similar way. \end{proof}
\begin{prop}
\label{pro: Quo by 5}Let be $G=\mathbb{Z}/5\mathbb{Z}$. The singularities
of $S/G$ are $4A_{5,2}$, or $A_{5,1}+2A_{5,2}+A_{5,4}$ or $2A_{5,1}+2A_{5,4}$.
The invariants of the surface $Z$ are respectively: \[
\begin{array}{c}
K_{Z}^{2}=0,\, c_{2}=12,\\
K_{Z}^{2}=-1,\, c_{2}=13,\\
K_{Z}^{2}=-2,\, c_{2}=14,\end{array}\]
 and in any case $q=p_{g}=0$. \end{prop}
\begin{rem}
1) In Proposition \ref{pro:cyclique d'ordre 10} below, we give an
example of a surface such that the quotient by an order $5$ automorphism
has $2A_{5,1}+2A_{5,4}$ singularities.\\
 2) For the same reasons as for $(\mathbb{Z}/3\mathbb{Z})^{2}$
(see Proposition \ref{pro:Quotient by 3 times 3}), there is no fake
quadric $X$ with $(\mathbb{Z}/5\mathbb{Z})^{2}\subset\aut X$.\end{rem}
\begin{proof}
Using the notations of Proposition \ref{prop Zhang formula}, the
number of fixed points $r_{1}+r_{2}+r_{3}+r_{4}$ equals $4$. As
$e(S/\sigma)=\frac{1}{5}(4+(5-1).4)=4$ Zhang's formula yields\[
(a_{1},...,a_{4})=(0,\frac{1}{4},\frac{1}{4},\frac{1}{2}),\]
 with\[
\sum4a_{i}r_{i}=r_{2}+r_{3}+2r_{4}=4.\]
 Thus $r_{1}=r_{4}$. Therefore the possibilities for $(r_{1},r_{2},r_{3},r_{4})$
are $(0,i,j,0)$ with $i+j=4$, or $(1,i,j,1)$ with $i+j=2$ or $(2,0,0,2)$.
The singularities on the quotient are respectively:\[
\begin{array}{c}
4A_{5,2},\\
A_{5,1}+2A_{5,2}+A_{5,4},\\
2A_{5,1}+2A_{4}.\end{array}\]
 A singularity $A_{5,i}$ ($i=1,..4$) contributes (respectively)
\[
-\frac{9}{5},-\frac{2}{5},-\frac{2}{5},0\]
 to $K_{Z}^{2}$. Thus the self-intersection number is\[
K_{Z}^{2}=\frac{1}{5}(8-9r_{1}-2(r_{2}+r_{3})),\]
 and according to the possible tuples $(r_{1},\ldots,r_{4})$ as above:
$K_{Z}^{2}=0$, $K_{Z}^{2}=-1$ and $K_{Z}^{2}=-2$. As $e(S/G)=4$,
we get $c_{2}=12,13$ or $14$ according to the three possible singular
loci.

Let us justify our computation of $K_{Z}^{2}$. A $A_{5,1}$-singularity
is resolved by a $(-5)$-curve $C_{5}$, thus we have to add $-\frac{3}{5}C_{5}$
to the canonical divisor. This contributes $(-\frac{3}{5}C_{5})^{2}=-\frac{9}{5}$
to $K_{Z}^{2}$. On the other hand, a $A_{5,2}$-singularity is resolved
by a chain of two curves $C_{2},C_{3}$ with $C_{k}^{2}=-k$. We have
to add $-\frac{2}{5}C_{3}-\frac{1}{5}C_{2}$ to $\pi^{*}K_{X}$, and
the contribution to $K_{Z}^{2}$ is \[
(\frac{2}{5}C_{3}+\frac{1}{5}C_{2})^{2}=-\frac{2}{5}.\]
 Finally, note that $A_{5,3}=A_{5,2}$ and that the $A_{5,4}$-singularity
doesn't contribute to $K_{Z}^{2}$. \end{proof}
\begin{prop}
\label{Prop. Z/6Z} If $G=\mathbb{Z}/6\mathbb{Z}$, then $S/G$ has
singularities $2A_{6,1}+2A_{6,5}$. The minimal resolution $Z$ has
invariants: \[
K_{Z}^{2}=-4,\, c_{2}=16,\, q=p_{g}=0.\]
\end{prop}
\begin{proof}
Let $s$ be a fixed point of an order 6 automorphism $\sigma$. Let
$\alpha$ be a primitive third root of unity. By Lemma \ref{lemme restriction sur les automor},
the action of $\sigma$ on $T_{S,s}$ has eigenvalues $(-\alpha,(-\alpha)^{a})$
or $(-\alpha^{2},(-\alpha^{2})^{a})$ with $a=1\,\mbox{or}\,5$. Let
$r_{1},\, r_{2}$ and $r_{3}$ be respectively the number of fixed
points of $\sigma$ with eigenvalues $(-\alpha,-\alpha)$, $(-\alpha^{2},-\alpha^{2})$
and $(-\alpha,-\alpha^{5})$. Lefschetz fixed point formula (Proposition
\ref{pro:Lefschetz formula}) implies the relation \[
\frac{r_{1}}{(1+\alpha)^{2}}+\frac{r_{2}}{(1+\alpha^{2})^{2}}+r_{3}=1,\]
 therefore $r_{1}=r_{2}$ and $-r_{1}+r_{3}=1$. By Corollary \ref{cor:lefschetz top},
$\sigma$ has $2$ or $4$ fixed points. The only possibility for
$(r_{1},r_{3})$ is therefore $(1,2)$. The singularities are $2A_{6,1}+2A_{6,5}$
and the minimal resolution $Z$ of $S/\sigma$ has $K_{Z}^{2}=\frac{8}{6}-2.\frac{8}{3}=-4$.
Moreover $e(Z)=\frac{1}{6}(4+5\cdot4)+2+2\cdot5=16$.
\end{proof}
Let us study the case $G=\mathbb{Z}/8\mathbb{Z}$.
\begin{prop}
Let $\sigma$ be an order $8$ element acting on $S$. The singularities
of $S/\sigma$ are $2A_{8,3}+2A_{8,5}$. The resolution $Z$ of the
quotient surface is a surface with \[
K_{Z}^{2}=-2,\, c_{2}(Z)=14,\, q=p_{g}=0.\]
\end{prop}
\begin{proof}
Let $p$ be a fixed point of $\sigma$ and Let $\xi_{(p)}$ be a primitive $8$-th root 
of unity such that $\sigma$ acts on $T_{S,p}$ with eigenvalues $\xi_{(p)},\xi_{(p)}^{q_{p}}$ for
$q_{p}\in\{0,\cdots,7\}$. Since there are no reflexions, we must have
$\xi_{(p)}^{j}\not=1$ and $\xi_{(p)}^{jq_{p}}\not=1$ for $j=1,\cdots,7$,
thus $q_{p}$ is prime to $2$ and $q_{p}\in\{1,3,5,7\}$.Let $a_{1},a_{3},a_{5}$
and $a_{7}$ be the number of fixed points with $q_{p}=1,3,5$ or
$7$ respectively. We have $\sum a_{i}=2$ or $4$. By summing over
the powers $\sigma^{k}$ for $k=1\dots7$ in the formula of the
Holomorphic Lefschetz Theorem, we get \[
7=\sum_{p\in S^{\sigma}}\sum_{k=1}^{k=7}\frac{1}{\det(1-d\sigma^{k}|T_{S,p})},\]
 and thus \[
7=\sum_{u=0}^{u=3}\sum_{k=1}^{k=7}\frac{a_{2u+1}}{(1-\xi^{k})(1-\xi^{k(2u+1)})}=\frac{7}{4}a_{1}+\frac{5}{4}a_{3}+\frac{9}{4}a_{5}+\frac{21}{4}a_{7}.\]
 The possibilities for $(a_{1},\dots,a_{4})$ are $(4,0,0,0)$, $(2,1,1,0)$,
$(1,0,0,1)$ and $(0,2,2,0)$.

For $t^{2}$ of order $4$, we have seen that the singularities of
$S/\sigma^{2}$ are $2A_{4,1}+2A_{4,3}$ or $A_{1}+2A_{4,3}$. Thus
the only possibility for $(a_{1},\dots,a_{4})$ is $(0,2,2,0)$, and
the singularities of $S/\sigma$ are $2A_{8,3}+2A_{8,5}$. The Euler
number of $S/\sigma$ is \[
e(S/\sigma)=\frac{1}{8}(4+7\cdot4)=4.\]
 Since $\frac{8}{3}=3-\frac{1}{3}$ and $\frac{8}{5}=2-\frac{1}{3-\frac{1}{2}}$
we get\[
e(Z)=4+2\cdot2+2\cdot3=14.\]
 It is easy to chek that a singularity $A_{8,3}$ decreases $K_{Z}^{2}$
by $1$ and a singularity $A_{8,5}$ decreases $K^{2}$ by $\frac{1}{2}$,
thus we obtain: $K_{Z}^{2}=\frac{8}{8}-2\cdot1-2\cdot\frac{1}{2}=-2$. 
\end{proof}

\begin{prop}
\label{pro:cyclique d'ordre 10}Let $S$ be a fake quadric with $G=\mathbb{Z}/10\mathbb{Z}\subset\aut(S)$.
The singularities of the quotient surface $X=S/G$ are $ $$2A_{10,1}+2A_{10,9}$.
The resolution $Z$ has the invariants:\[
K^{2}=-12,\, c_{2}=24,q=p_{g}=0.\]
\end{prop}

\begin{proof}
Let $\sigma$ be an automorphism of order $10$ acting on $S$. It
has $2$ or $4$ fixed points. As the involution $\sigma^{5}$ has
$4$ fixed points, $\sigma$ cannot have $2$ fixed points. Therefore:
\[
e(S/G)=\frac{1}{10}(4+(10-1).4)=4.\]
 Let $\xi$ be a primitive $5^{th}$-root of unity and $p$ a fixed
point. There exist $a=a(p)$ and $b=b(p)$ integers invertible mod
$5$ such that the action of $\sigma$ on $T_{S,p}$ has eigenvalues
$(-\xi^{a},-\xi{}^{ba})$. The Lefschetz holomorphic fixed point formula
yields\[
1=\sum_{p\in S^{\sigma}}\frac{1}{(1+\xi^{a})(1+\xi^{ab})}.\]
 For $b=1,2,3,4$, the sum $c(b)=\sum_{a=1}^{a=4}\frac{1}{(1+\xi^{a})(1+\xi^{ab})}$
is equal to $-4,1,1,6$, respectively. Recall again that $A_{10,3}=A_{10,7}$.
For $k\in\{1,3,9\}$, let $r_{k}$ be the number of points in $S^{\sigma}$
giving a $A_{10,k}$ singularity. By summing the Lefschetz fixed point
\[
4=-4r_{1}+r_{3}+6r_{9}.\]
 Taking care of the relation $r_{1}+r_{3}+r_{9}=4$, we have the following
possibilities for $(r_{1},r_{3},r_{9})$: $(0,4,0)$, $(1,2,1)$ and
$(2,0,2)$.

The resolution of a $A_{10,3}$-singularity is a chain of $3$ curves
$C_{2},C_{2}',C_{4}$ with intersection numbers $(-2)-(-2)-(-4)$.
We have to add $-\frac{1}{5}(C_{2}+C_{2}'+C_{4})$ to $\pi^{*}K_{S/G}$.
Each singularity contributes $(-\frac{1}{5}(C_{2}+C_{2}'+C_{4}))^{2}=-\frac{6}{5}$
to $K_{Z}^{2}$.\\
 Similarly, the resolution of a $A_{10,1}$-singularity is $(-10)$-curve
$C_{10}$. A $A_{10,1}$-singularity decreases $K_{S/G}^{2}$ by $(\frac{-8}{10}C_{10})^{2}=-\frac{32}{5}$.\\
 When the singularities of $S/G$ are respectively $4A_{10,3},\, A_{10,1}+2A_{10,3}+A_{10,9}$
and $2A_{10,1}+2A_{10,9}$, we have: $K_{Z}^{2}=\frac{8}{10}-4\frac{6}{5}=-4$,
$ $$K_{Z}^{2}=\frac{8}{10}-\frac{32}{5}-2.\frac{6}{5}-0=-8$ and
$K_{Z}^{2}=\frac{8}{10}-2\frac{32}{5}=-12$. The Euler number of $Z$
is respectively $4+4\cdot2=12$, $4+1+2\cdot2+9=18$ and $4+2+2\cdot9=24$.
Only the last case is possible because $12$ has to divide $K_{Z}^{2}+e(Z)$. \end{proof}

\begin{prop}
\label{pro:quotient by Klein 4 grup}Let be $G=(\mathbb{Z}/2\mathbb{Z})^{2}$.
The quotient surface $X=S/G$ contains $6\ A_{1}$ singularities.
The surface $Z$ is minimal of general type and has the invariants:
\[
K_{Z}^{2}=2,\, c_{2}=10,\, q=p_{g}=0.\]
\end{prop}
\begin{proof}
A faithful representation of $G$ on a $2$-dimensional space contains
reflections, therefore by Lemma \ref{lemme restriction sur les automor},
there are no points fixed by the whole $G$. The group $G$ contains
$3$ involutions. Each of these involutions has $4$ isolated fixed
points whose image in $X$ are $2A_{1}$ singularities. Thus there
are $6A_{1}$ singularities on $X=S/G$ and we have \[
e(Z)=e(S/G)+6=\frac{1}{4}(4+12)+6=10.\]
 Moreover, $K_{Z}=\pi^{*}K_{S/G}$ is nef and $K_{S/G}^{2}=K_{S}^{2}/4=2$.
By Lemma \ref{lemme q=00003D00003D00003Dpg=00003D00003D00003D0},
we have $q=p_{g}=0$.\end{proof}

\begin{rem}
a) As Miles Reid pointed out to us, a minimal surface of general type
with $c_{1}^{2}=2c_{2}=8,\, p_{g}=0$ and automorphism group containing
$G=(\mathbb{Z}/2\mathbb{Z})^{3}$ such that each involution has only
isolated points must deform, therefore $(\mathbb{Z}/2\mathbb{Z})^{3}$
cannot be a subgroup of the automorphism group of a quaternionic fake
quadric which are rigid surfaces.
\\
b) For $G=\mathbb{Z}/4\mathbb{Z} \times \mathbb{Z}/2\mathbb{Z}$, the quotient surface $S/G$ has singularities $2A_1 +2A_3$ and the desingularisation $Z$ has invariants $K_{Z}^{2}=1,c_{2}=11,q=p_{g}=0$. We do not know if a fake quadric $S$ with such automorphism subgroup exists.
\end{rem}

\begin{prop}
\label{pro:Let-be-Dihedral D4}Let be $G=\mathbb{D}_{4}$ acting on
the fake quadric $S$. The singularities of $S/G$ are $4A_{1}+A_{4,3}+A_{4,1}$.
The resolution $Z$ of the quotient surface has invariants: \[
K_{Z}^{2}=0,\, c_{2}(Z)=12,\, q=p_{g}=0.\]
 The elements of order 4 in $\mathbb{D}_{4}$ have $4$ fixed points.\end{prop}
\begin{proof}
Let be $t,a$ be generators of $\mathbb{D}_{4}$ such that $t^{4}=1$,
$a^{2}=1$ and $at=t^{3}a$. The elements of order $4$ are $t$ and
$t^{3}$. The elements of order $2$ are $a,ta,t^{2}a,t^{3}a$ and
$t^{2}$.

There cannot be a point of $S$ that is fixed by the whole group $G$
because any faithful $2$-dimensional representation of $G$ contains
a reflection $(x,y)\to(x,-y)$ and thus such a point would lie on a curve fixed 
by an involution. But an automorphism of $S$ has only isolated fixed points.

First case: let us suppose that $t$ has $4$ fixed points : $Fix(t)=\{p_{1},ap_{1},p_{2},ap_{2}\}$.
The Euler number of $S/G$ is \[
e(S/G)=\frac{1}{8}(4+(2-1)(4\cdot4)+(4-1)4)=4.\]
 The singularities on $S/G$ are $4A_{1}+A_{4,3}+A_{4,1}$ thus \[
e(Z)=4+4+3+1=12.\]
 Moreover $K_{Z}^{2}=\frac{8}{8}+(-\frac{1}{2})^{2}(-4)=0.$

Second case: suppose that $t$ has $2$ fixed points : $Fix(t)=\{p_{1},ap_{1}\}$.
The Euler number of $e(S/G)$ would be \[
\frac{1}{8}(4+(2-1)(18)+(4-1)2)=\frac{7}{2}\]
 but this is not an integer. \end{proof}
\begin{prop}
Suppose that the dihedral group $\mathbb{D}_{8}$ of order $16$ acts
on fake quadric $S$. The singularities of $S/\mathbb{D}_{8}$ are
$4A_{1}+A_{8,3}+A_{8,5}$. The resolution $Z$ of the quotient surface
has invariants: \[
K_{Z}^{2}=-1,\, c_{2}(Z)=13,\, q=p_{g}=0.\]
\end{prop}
\begin{proof}
Let be $t,a$ be generators of $\mathbb{D}_{8}$ such that $t^{8}=a^{2}=1$
and $at=t^{7}a$. Order $8$ elements in $G$ are $t,t^{3},t^{5},t^{7}$,
order $4$ elements are $t^{2},\, t^{6}$, order $2$ elements are
$a,ta,t^{2}a,t^{3}a,t^{4}a,t^{5}a,t^{6}a,t^{7}a$ and $t^{4}$.

By the discussion on order $8$ elements, $t$ has $4$ fixed points $Fix(t)=\{p_{1},ap_{1},p_{2},ap_{2}\}$.
Let $p$ be a fixed point of an involution $\sigma\not=t^{4}$. The
orbit of $p$ under $G$ has $8$ elements, each is a fixed point
of an involution $\not=t^{4}$. The quotient surface has $\frac{1}{8}\cdot8\cdot4A_{1}+A_{8,3}+A_{8,5}$
singularities. We have\[
e(S/G)=\frac{1}{16}(4+1\cdot(8\cdot4)+7\cdot4)=4\]
 and $e(Z)=4+4+2+3=13$. Moreover $K_{Z}^{2}=\frac{8}{16}-1-\frac{1}{2}=-1$.
\end{proof}

\section{Reconstruction of a surface knowing its quotient.}

In \cite{Miyaoka}, Miyaoka gives a bound on the number of disjoint
$(-2)$-curves on a minimal smooth surface $Y$. This implies in particular
that if $c_{1}^{2}=4$ or $2$ and $\chi(\mathcal{O}_{Y})=1$, there
are at most $4$ and $6$ such curves respectively. The surfaces we
obtained as quotient of quaternionic fake quadrics reach that bound.
For the cases $c_{1}^{2}=2$ these surfaces seems to be the first
known ones with this property.

In \cite{Dolgachev} Dolgachev, Mendes Lopes, Pardini study rational
surfaces with the maximal number of $(-2)$-curves. For that aim they
use and developpe the theory of $(\mathbb{Z}/2\mathbb{Z})^{n}$-covers
ramified over $A_{1}$ singularities. Using their results, we obtain: 
\begin{prop}
\label{pro:if enought nodes}Let $Y$ be a smooth minimal surface
of general type with $q=p_{g}=0$ and $_{2}Pic(Y)=0$. \\
 a) If $c_{1}(Y)^{2}=4$, $c_{2}(Y)=8$ and $Y$ contains $4$
disjoint $(-2)$-curves $C_{1},\dots,C_{4}$, then there exist a double
cover of $Y$ ramified over the curves $C_{i}$. The minimal model
of this covering has invariants $c_{1}^{2}=2c_{2}=8$ and $q\leq1$.
\\
 b) If $c_{1}(Y)^{2}=2$, $c_{2}(Y)=10$ and $Y$ contains $6$
disjoint $(-2)$-curves $C_{1},\dots,C_{6}$, then there exist a bi-double
cover of $Y$ ramified over the curves $C_{i}$. The minimal model
of this covering has invariants $c_{1}^{2}=2c_{2}=8$. 
\end{prop}
Let $\mathbb{F}_{2}$ be the field with $2$ elements. Let be $C_{1},\dots,C_{k}$
be $k$ $(-2)$-curves on a smooth surface $Y$. Let \[
\psi:\mathbb{F}_{2}{}^{k}\to Pic(Y)\otimes\mathbb{F}_{2}\]
 be the homomorphism sending $v=(v_{1},\dots,v_{k})$ to $\sum v_{i}C_{i}$.
We say that the curve $C_{j}$ appears in the kernel $\ker\psi$ if
there is a vector $v=(v_{1},\dots,v_{k})$ in $\ker\psi$ such that
$v_{j}=1$. For $v$ is in $\ker\psi$ we denote by $L_{v}$ an element
of $Pic(Y)$ such that $2L_{v}=\sum v_{i}C_{i}$ (we sometimes identify
elements of $\mathbb{F}_{2}$ to $0,1$ in $\mathbb{Z}$). We have: 
\begin{prop}
\label{pro:Galois group}(\cite{Dolgachev}, Proposition 2.3). Suppose
that $_{2}Pic(Y)$ is zero. There exists a unique smooth connected
Galois cover $\pi:Z\to Y$ such that the Galois group of $\pi$ is
$G=Hom(\ker\psi,\mathbb{G}_{m})$, the branch locus of $\pi$ is the
union of the $C_{i}$ appearing in $\ker\psi$ and the surface $\bar{Z}$
obtained by contracting the $(-1)$-curves over the $(-2)$-curves
in $Y$ has invariants: \[
K_{\bar{Z}}^{2}=2^{r}K_{Y}^{2}\, c_{2}(\bar{Z})=\,\chi(\mathcal{O}_{\bar{Z}})=2^{r}\chi(\mathcal{O}_{Y})-k2^{r-3},\,\kappa(\bar{Z})=\kappa(Y)\]
 where $r=\dim V$.\end{prop}
\begin{proof}
(Of Proposition \ref{pro:if enought nodes}). We have to prove
that for our surface $Y$, $\ker\psi$ has the requied dimension and
that all the curves appear in $\ker\psi$. For $c_{1}^{2}(Y)=4$
and $2$, we have $b_{2}(Y)=h^{1,1}(Y)=6$ and $8$ respectively.
As we supposed that $_{2}Pic(Y)=0$, the space $Pic(Y)\otimes\mathbb{F}_{2}$
is $h^{1,1}$ dimensional. As $p_{g}=0$, it has moreover a non-degenerate
intersection pairing and therefore the dimension of a totally isotropic
space in $Pic(Y)\otimes\mathbb{F}_{2}$ is at most $\left[\frac{h^{1,1}}{2}\right]=$
$3,\,\mbox{or }4$ dimensional respectively. The image of $\psi$
is the totally isotropic space generated by the curves $C_{i}$, therefore
the dimension $r$ of $\ker\psi$ is at least $1$ and $2$ respectively.
\\
 A smooth double cover of a surface with $n$ nodes can exist only
if $n$ is divisible by $4$ (see \cite{Dolgachev}). Therefore the
vectors $v=(v_{1},\dots,v_{k})$ in $\ker\psi$ (of dimension $\leq7$)
have weight $4$ i.e. the number of indices $j$ such that $v_{j}=1$
is $4$.\\
 In case a), $\ker\psi$ is one dimensional, generated by $w_{1}=(1,1,1,1)$.
For b), as every vector in $\ker\psi$ has weight $4$, by \cite{Beauville}
Lemme 1, we have $k\geq2^{r}-1$ and thus $r\leq2$ and $r\leq3$
respectively. Moreover, it is easy to check that in the case b), the
space $\ker\psi$ is (up to permutation of the basis vectors) generated
by $w_{1}=(1,1,1,1,0,0)$ and $w_{2}=(1,1,0,0,1,1)$. 
\end{proof}
Let us give a bound on the irregularity: 
\begin{lem}
Let $Y$ be a surface of general type with $\chi=1$ and $q=0$ containing
a $2$-divisible set of $4\ (-2)$-curves. Let $Y'\to Y$ be the double
cover. Then $q(Y')\leq1$.\end{lem}
\begin{proof}
As $q(Y)=0$, the involution $\sigma$ on $Y'$ given by the cover
$Y'\to Y$ acts as multiplication by $-1$ on $H^{0}(Y',\Omega_{Y'})$.
Therefore, $\sigma$ acts trivialy on $\wedge^{2}H^{0}(Y',\Omega_{Y'})$.
As $p_{g}(Y)=0$, the map $\wedge^{2}H^{0}(Y',\Omega_{Y'})\to H^{0}(Y',\wedge^{2}\Omega_{Y'})$
must be $0$. Let $Y'\to Y''$ be the blow-down map of the $4\ (-1)$-curves
over the $4$ nodal curves of $Y$. If $q(Y'')\geq1$, Castelnuovo-De
Franchis Theorem implies that there is a fibration onto a curve $B$
of genus $q(Y'')$. By \cite{Zucconi}, we get that $q(Y'')\leq2$
and if $q(Y'')=2$, then $Y''$ is an \'etale bundle of genus $2$
fibers onto a genus $2$ curve $B$ and $K_{Y''}^{2}=8$. In that
case, there is a commutative diagram\[
\begin{array}{ccc}
Y'' & \to & X\\
\downarrow &  & \downarrow\\
B & \to & \mathbb{P}^{1}\end{array}\]
 where the vertical maps are genus 2 fibrations and $X$ is
the surface obtained by contracting the $4(-2)$-curves on $Y$. This
diagram is obtained from $B\to\mathbb{P}^{1}$ by taking base change
and normalizing. Since $Y''\to X$ is unramified in codimension $1$,
the $6$ fibers of $X\to\mathbb{P}^{1}$ occurring at the 6 branch
points of $B\to\mathbb{P}^{1}$ are double. Since $X$ has only 4
singular points, $X\to\mathbb{P}^{1}$ has at least two double fibers
contained in the smooth locus of X, but a multiple fiber on a genus
2 fibration cannot exists (because of the adjunction formula). Thus
$q\leq1$. 
\end{proof}
Let us now consider a smooth minimal surface of general type $Z$
with $K^{2}=2,\, c_{2}=10$, $q=p_{g}=0$ such that there is a birational
map onto a surface $Y$ with singularities $2A_{3,1}+2A_{3,2}$. 
\begin{prop}
Suppose that $_{3}Pic(Z)=0$. There exists a smooth triple cover $X$
of $Y$ ramified precisely over the singularities of $Y$. The surface
$X$ is of general type and has invariants $c_{1}^{2}=2c_{2}=8$.\end{prop}
\begin{proof}
Let $D_{1},D_{2}$ be the $(-3)$-curves over the $2$ singularities
$A_{3,1}$ and let $D_{3},\dots,D_{6}$ be the $(-2)$ curves over
the singularities $A_{3,2}$, with indices satisfying: $D_{3}D_{4}=D_{5}D_{6}=1$.
Let $W\rightarrow Y$ be the blow up at the intersection points of
$D_{3},D_{4}$ and of $D_{5},D_{6}.$ Let $C_{1},\dots,C_{6}$ be
the strict transforms of the $D_{i}$ in $W$. Let\[
\psi:\mathbb{F}_{3}{}^{6}\to Pic(W)\otimes\mathbb{F}_{3}=H^{2}(W,\mathbb{F}_{3})\]
 be the homomorphism sending $v=(v_{1},\dots,v_{k})$ to $\sum v_{i}C_{i}$.
The image of $\psi$ is a totally isotropic subspace in $H^{2}(W,\mathbb{F}_{3})$.
As $b_{2}(W)=10$, this image is at most $5$ dimensional and therefore
$\dim\ker\psi\geq1$. Let be $v=(v_{1},\dots,v_{6})\in\ker\psi$,
$v\not=0$. We choose the representatives of $\mathbb{F}_{3}$ in
$\{0,1,2\}$. There exist a unique invertible sheaf $L$ such that
\[
3L=\sum v_{i}C_{i}.\]
 Let $T$ be the triple cover of $W$ ramified over the $r$ curves
$C_{i}$ such that $v_{i}\not=0$. The surface $T$ is smooth outside
the curves $C_{i}$ with $v_{i}=2$. Let $R$ be the minimal resolution
of $T$ and let $f:R\to W$ be the composite map. By \cite{Urzua},
Propositions 2.2, 4.1 and $4.3$, the invariants of $R$ are: \[
K_{R}=_{num}f^{*}(K_{W}+\frac{2}{3}\Sigma),\, c_{2}(R)=3c_{2}(W)-4r,\,\chi(\mathcal{O}_{R})=3\chi(\mathcal{O}_{W})-\frac{r}{3},\]
 where $\Sigma$ is the sum of the $r$ curves $C_{i}$ such that
$v_{i}\not=0$. Therefore $r=3$ or $6$ and \[
K_{R}^{2}=0,\, c_{2}(R)=36-4r,\,\chi(\mathcal{O}_{W})=3-\frac{r}{3}.\]
 As there are at least $3$ curves $C_{i}$ in the branch locus, one
of the curves $C_{3},\dots,C_{6}$ is in that branch locus. Say it
is $C_{3}$. Let $E$ be the exceptional curve going through $C_{3}$.
As $C_{3}E=C_{4}E=1$ and $E\sum v_{i}C_{i}$ is divisible by $3$,
it forces $C_{4}$ to be also in the branch locus and thus $r=6$
(and $\dim\ker\psi=1$). The inverse image of the $6$ $(-3)$-curves
are $(-1)$-curves. By the formula giving $K_{R},$ the inverse image
of the two exceptional curves are $(-3)$-curves meeting two $(-1)$-curves.
We can therefore effectuate $8$ blow-downs and we obtain a fake quadric.
It has general type because $Y$ has general type, it is minimal because
the quotient of a fake plane by an order $3$ automorphism with $4$
isolated fixed points has $4A_{2}$ singularities. \end{proof}

\vspace{0.2in}
 {\large \setlength{\parindent}{0.5in}
 
 \noindent  
 Amir D\v{z}ambi\'{c},}{\large \par}

{\large Johann Wolfgang Goethe Universität, Institut für Mathematik,}{\large \par}

{\large Robert-Mayer-Str. 6-8,}{\large \par}

{\large 60325 Frankfurt am Main,}{\large \par}

{\large Germany}{\large \par}

\texttt{\large dzambic@math.uni-frankfurt.de}{\large \par}

{\large \vspace{0.2in}
 \setlength{\parindent}{0.5in}}{\large \par}

{\large Xavier Roulleau,}{\large \par}

{\large Laboratoire de Mathématiques et Applications,}{\large \par}

{\large Universit\'e de Poitiers,}{\large \par}

{\large Téléport 2 - BP 30179 - }{\large \par}

{\large 86962 Futuroscope Chasseneuil }{\large \par}

{\large France}{\large \par}

\texttt{\large roulleau@math.}{\large univ-poitiers.fr{} } 

\begin{thebibliography}{28}
\bibitem{Atiyah}\textsc{Atiyah, M.~F., Singer, I.~M.} \newblock
The index of elliptic operators. I, II, III. \newblock {\em Ann.
of Math. (2) 87\/} (1968), 484--530, 531--545, 546--604.

\bibitem{Barth} \textsc{Barth, W.~P., Hulek, K., Peters, C. A.~M.,
Van~de Ven, A.} \newblock { Compact complex surfaces}, 2nd .ed
\newblock Ergebnisse der Mathematik und ihrer Grenzgebiete. 3.~Folge,
Bd.~4. Springer-Verlag, Berlin, 2004.

\bibitem{bauercatanesegrunewald} \textsc{Bauer, I.C., Catanese, F.,
Grunewald, F.} \newblock The classification of surfaces with {$p_{g}=q=0$}
isogenous to a product of curves. \newblock {\em Pure Appl. Math.
Q. 4}, 2, part 1 (2008), 547--586.

\bibitem{Beauville} \textsc{Beauville A.} \newblock Sur le nombre
maximum de points doubles d'une surface dans $\mathbb{P}^{3}$ ($\mu(5)=31$).
\newblock In: {\em Journ\'ees de G\'eometrie Alg\'ebrique d'Angers},
pp. 207-215, Sijthoff Noordhoff, Alphen aan den Rijn-Germantown, 1980.

\bibitem{borel} \textsc{Borel, A.} \newblock Commensurability classes
and volumes of hyperbolic {$3$}-manifolds. \newblock {\em Ann.
Scuola Norm. Sup. Pisa Cl. Sci. (4) 8}, 1 (1981), 1--33.

\bibitem{Dolgachev} \textsc{Dolgachev, I., Mendes Lopes, M., Pardini,
R.} \newblock Rational surfaces with many nodes. \newblock {\em
Compositio Math.}, 132 (2002), no. 3, 349--363.

\bibitem{dz11} \textsc{{D}\v{z}ambi\'{c}, A.} \newblock {F}ake
quadrics from irreducible lattices acting on the product of upper
half planes. \newblock In preparation.

\bibitem{Catanese} \textsc{Catanese F., Di Scala A.} \newblock Characterization
of Varieties whose Universal Cover is the Polydisk or a Tube Domain.
\newblock ArXiv:1011.6544

\bibitem{Granath} \textsc{Granath, H.} \newblock { On quaternionic
Shimura surfaces.} \newblock {\em Chalmers Tekniska Hogskola (Sweden)},
2002, Thesis (PhD).

\bibitem{Inose} \textsc{Inose, H., Mizukami}, M. Rational equivalence
of 0-cycles on some surfaces of general type with pg=0. \newblock
{\em Math. Ann. } 244 (1979), no. 3, 205--217.

\bibitem{Keum11} \textsc{Keum, J.} Toward a geometric construction
of fake projective plane, \newblock {\em Rend. Lincei Mat. Appl.}
22 (2011), 1--19

\bibitem{Keum1} \textsc{Keum, J.} Projective surfaces with many nodes.
Algebraic geometry in East Asia Seoul 2008, 245--257, \newblock {\em
Adv. Stud. Pure Math.} 60, Math. Soc. Japan, Tokyo, 2010.

\bibitem{Keum} \textsc{Keum, J.} \newblock Quotients of fake projective
planes. \newblock {\em Geom. Topol. 12}, (2008), no. 4, 2497--2515.

\bibitem{Keum2} \textsc{Keum, J.} \newblock A fake projective plane
with an order 7 automorphism. \newblock {\em Topology} 45 (2006),
no. 5, 919--927

\bibitem{MatsushimaShimura} \textsc{Matsushima, Y., and Shimura,
G.} \newblock On the cohomology groups attached to certain vectorto
valued differential forms on the product of the upper half planes.
\newblock {\em Ann. of Math.} (2) 78\/ (1963), 417--449.

\bibitem{Mendes} \textsc{Mendes Lopes M., Pardini R.} The bicanonical
map of surfaces with $p_{g}=0$ and $K^{2}\ge7$. \newblock {\em
Bull. London Math. Soc.} 33 (3) (2001), 265--274.

\bibitem{Miyaoka} \textsc{Miyaoka, Y.} \newblock The maximal number
of quotient singularities on surfaces with given numerical invariants.to
\newblock { \em Math. Ann}. 268, (1984), 159-171.

\bibitem{Pardini} \textsc{Pardini} The classification of double planes
of general type with $p_{g}=0$ and $K^{2}=8$. \newblock {\em
J. of Algebra} 259 (2003), 95--118.
toto
\bibitem{Riehm} \textsc{Riehm, C.} \newblock Norm-1 group of a $\mathfrak{p}$-adic
division algebra. \newblock {\em Am. Journ. of Math. 92}, 2 (1970),
499--523.

\bibitem{Roulleau11} \textsc{Roulleau X.} \newblock Quotients of
Fano surfaces. \newblock {\em Rend. Lincei Mat. Appl. 23} (2012),
1--25.

\bibitem{Shavel78} \textsc{Shavel, I.H.} \newblock A class of algebraic
surfaces of general type constructed from quaternion algebras. \newblock
{\em Pacific J. Math.} 76, 1 (1978), 221--245.


\bibitem{shim:zeta} \textsc{Shimura, G.} \newblock Construction
of class fields and zeta functions of algebraic curves. \newblock
{\em Ann. of Math.} (2) 85\/ (1967), 58--159.

\bibitem{Urzua} \textsc{Urzua, G.} \newblock Arrangements of curves
and algebraic surfaces. \newblock {\em J. Algebraic Geom.} 19,
2 (2010), 335--365.

\bibitem{vign2} \textsc{Vign\'eras, M.-F.} \newblock {Arithm\'etique
des alg\`ebres de quaternions}. \newblock {\em Lecture Notes in
Mathematics}, 800, Springer, Berlin, 1980.

\bibitem{Zhang} \textsc{Zhang, D.-Q.} \newblock Automorphisms of
finite order on rational surfaces. With an appendix by I. Dolgachev.
\newblock {\em J. Algebra} 238, 2 (2001), 560--589.

\bibitem{Zucconi} \textsc{Zucconi, F.,} Surfaces with $p_{g}=q=2$
and an irrational pencil. \newblock {\em Canad. J. Math.} 55 (2003),
no. 3, 649--672.

\end{thebibliography}
\end{document}